\documentclass[11pt,twoside]{article}

\usepackage{epsfig,amsfonts,color}
\usepackage{amsmath}

\usepackage{amssymb, palatino, geometry,url}
\usepackage{algorithmic}
\usepackage[noresetcount,lined,boxed]{algorithm2e} 

\usepackage[colorlinks=true,linkcolor=blue,citecolor=blue,urlcolor=blue]{hyperref}

\usepackage{fancyhdr}
\newcommand{\be}{\begin{equation}}
\newcommand{\ee}{\end{equation}}
\newcommand{\bea}{\begin{eqnarray}}
\newcommand{\eea}{\end{eqnarray}}

\newcommand{\bvec}{\left(\begin{array}{c}}
\newcommand{\evec}{\end{array}\right)}
\newcommand{\bsub}{\begin{subequations}}
\newcommand{\esub}{\end{subequations}}

\usepackage{lineno}

\usepackage{tabularx} 
\usepackage{svg} 
\usepackage{makecell} 
\usepackage{amsthm} 
\usepackage{verbatim} 

\providecommand{\e}[1]{\ensuremath{\times 10^{#1}}} 

\newtheorem{theorem}{Theorem} 

\theoremstyle{definition}
\newtheorem{definition}{Definition}
\theoremstyle{remark}

\theoremstyle{corollary}
\newtheorem*{corollary}{Corollary}
\theoremstyle{lemma}

\begin{document}

\title{Economic Properties of Multi-Product Supply Chains}

\author{Philip A. Tominac and Victor M. Zavala\thanks{Corresponding Author: victor.zavala@wisc.edu}\\
 {\small Department of Chemical and Biological Engineering}\\
 {\small University of Wisconsin-Madison, 1415 Engineering Dr, Madison, WI 53706, USA}}
 \date{}
\maketitle

\begin{abstract}
We interpret multi-product supply chains (SCs) as coordinated markets; under this interpretation, a SC optimization problem is a market clearing problem that allocates resources and associated economic values (prices) to different stakeholders that bid into the market (suppliers, consumers, transportation, and processing technologies). The market interpretation allows us to establish fundamental properties that explain how physical resources (primal variables) and associated economic values (dual variables) flow in the SC. We use duality theory to explain why incentivizing markets by forcing stakeholder participation (e.g., by imposing demand satisfaction or service provision constraints) yields artificial price behavior, inefficient allocations, and economic losses. To overcome these issues, we explore market incentive mechanisms that use bids; here, we introduce the concept of a stakeholder graph (a product-based representation of a supply chain) and show that this representation allows us to naturally determine minimum bids that activate the market. These results provide guidelines to design SC formulations that properly remunerate stakeholders and to design policy that foster market transactions.  The results are illustrated using an urban waste management problem for a city of 100,000 residents. 
\end{abstract}


\section{Introduction}

Supply chain (SC) optimization is an essential industrial task that has seen significant research in recent years. In a multi-product supply chain setting,  raw materials are transported to processing sites to create valude-added products and these products are in turn transported to final customers. The tenet of most existing SC models is that there exists a {\em central entity} that coordinates all assets in the supply chain (supply, consumption, transportation, and processing) in order to maximize total profit. In other words, the SC problem is interpreted as a central planning problem~\cite{pibernik2006centralised}. The implication of this setting is that the central planner dictates the allocation (distribution) of resources and associated economic value among assets. We direct the interested reader to several comprehensive reviews on the topic \cite{LIMA201678,GARCIA2015153,BARBOSA2014,PAPAGEORGIOU20091931}.

An alternative interpretation of a SC is one in which assets are owned by independent stakeholders that seek to strategically optimize their individual profits. Recently, we have observed that this interpretation matches that of a coordinated market \cite{Sampat2018B}. There is a rich body of literature concerning the use of coordinated markets for power systems~\cite{blumstein2002history,bohn1984optimal,hogan2002electricity,pritchard2010single,zavala2017stochastic}. In fact, most modern power networks are managed using coordinated markets. In such markets, stakeholders submit bids for power supply and consumption and for transportation (transmission) services to a non-profit entity known as the independent system operator (ISO). The ISO uses the bidding information to solve a network optimization problem (known as the market clearing problem) that determines power supply, demand, and transmission allocations that maximize the social welfare and that satisfy capacity constraints and network balance constraints (supply and demand matches at each location). Importantly, the market clearing problem has an equivalent dual representation that reveals the inherit value of power (clearing prices) at every network location. We thus have that the standard representation of the clearing problem (known as the primal representation) determines {\em physical} allocations while the dual representation determines {\em economic} allocations. An important result is that the primal-dual allocations obtained from the clearing problem correspond to those of a {\em competitive equilibrium} in which stakeholders maximize their individual profits. Importantly, this equilibrium is reached instantenously by the market clearing procedure; in a non-coordinated market, equilibrium is reached progressively over time (by peer-to-peer transactions). This desirable coordination behavior can help the market quickly respond to externalities (e.g., extreme weather or shortages of supplies). For instance, in the power grid, the coordination mechanism is essential to respond to extreme weather events and fluctuations of renewable power. The primal-dual allocations obtained from a coordinated market are also such that no stakeholder loses money and such that revenue is balanced in the system (financial gains of suppliers and service providers match financial payments of consumers). Sampat and co-workers recently generalized coordinated market formulations to account for multiple products, processing technologies, and transportation options \cite{Sampat2018B}. This work revealed that it is possible to construct multi-product SC models that inherit desirable economic properties of coordinated markets.

In this work, we exploit the coordinated market interpretation of SCs to derive new economic properties that explain how physical and economic resources are created and flow in the SC. Specifically, we propose a product-based representation of the SC (that we call a stakeholder graph) to  analyze how products and associated values flow from suppliers to consumers through intermidiate transportation and transformation steps. We use this representation to derive a procedure to determine minimum bids that activate a market. This approach can be used by policy makers to determine appropriate incentives that foster market transactions and investment. We use also duality theory to show that incentivizing markets by forcing stakeholder participation (as is done in typical SC settings) can induce revenue imbalances in the system, induce arbitrary price behavior, and can lead to financial losses for some stakeholders.  Our results provide guidelines to derive SC models  that induce desired economic behavior and to design proper incentives by policy makers (e.g., environmental agencies) that foster stakeholder participation. We apply these results to a waste SC arising from municipal solid waste disposal recycling problem. In this problem, the local government is interested in diverting waste from the landfill to recycling facilities~\cite{MALINAUSKAITE20172013,SOMPLAK2019118068,tisserant2017}. To this end, the proposed SC model has available material recycling facilities that offer a more expensive (but sustainable) means of waste disposal. We use our developments to determine bidding prices for recycled materials that are necessary to activate and sustain a recycling market and discuss the implications for policy makers. 

\section{Multi-Product Supply Chains as Coordinated Markets}\label{sec:model}

We represent a multi-product SC comprising a collection of independent stakeholders as a coordinated market. Market coordination brings together potential buyers (consumers) and sellers (suppliers) to resolve transactions in an efficient and competitive manner. Buyers and sellers make their monetary valuations of products (bids) known to the ISO who resolves transactions by solving a market clearing problem. Importantly, the only interest of the ISO in the market is the creation and balancing of economic value; the ISO only coordinates transactions and does not interfere with the objectives of buyers and sellers. In the case of consumers, a bid is the highest price that a consumer is willing to pay to receive a product. Sellers and service providers submit bids representing the lowest price at which they are willing to provide a product or service~\cite{Sampat2018A,Sampat2018B}. An important aspect of the clearing problem is that transactions between consumers and suppliers do not occur if bids do not generate wealth in the market, a property referred to as revenue adequacy. In simple settings, revenue adequacy implies that a consumer bid must be greater than the corresponding bid of a supplier for a transaction to occur. When no transactions occur, the market is said to be {\em dry}. 

In more realistic settings, product transportation within the SC and product transformation technologies are modeled under their own stakeholder classes. An SC is thus a collection of suppliers, transportation services, processing services, and consumers. In such a setting, it is expected that market transactions occur if a consumer bid for a given product ``beats" the collective bids of suppliers, transportation, and transformation  involved in the pathway that generates that product (i.e., the value of the product needs to be higher than the costs associated with supplies, transport, and processing involved in its creation). Establishing exact relationships a general SC setting is  challenging and is the subject of this work. We also highlight that bids should provide natural incentives for stakeholders to participate (or not) in the market. In other words, forcing stakeholder participation (e.g., by using supply and service provision constraints) is not allowed in the market. We will see in Section~\ref{sec: forced markets} that forcing participation can destroy desirable economic properties of a market.

We now present notation for the market clearing problem: SC stakeholders represented in the market are suppliers (denoted by $i\in\mathcal{S}$), consumers or demands (denoted by $j\in\mathcal{D}$), transportation providers (denoted by $l\in\mathcal{L}$), and transformation (processing) providers (denoted by $t\in\mathcal{T}$). The products exchanged in the system are indexed $p\in\mathcal{P}$ and the set of {\em geographical} nodes in the SC graph are $n\in\mathcal{N}$. To enable compact notation, the indices are always associated with the corresponding sets. The connections between nodes define the {\em supply chain graph}, which is a graph (a network) that connects stakeholders. We assume that the SC graph is fully connected (i.e., there exists a path from each node $n\in\mathcal{N}$ to every other $n'\in\mathcal{N}$). An important observation is that the SC graph uses a node-based representation (stakeholders are associated with specific nodes). Later we will see that there is an alternative (product-based) representation of the SC in which that captures how stakeholders are connected via products and how products and associated economic values flow in the system.

Suppliers $i\in\mathcal{S}$ each have a physical supply variable (flow) $s_{i}\in\mathbb{R}_{+}$, a bid capacity $\overline{s}_{i}\in\mathbb{R}_{+}$, a bid cost $\alpha^{s}_{i}\in\mathbb{R}$, and attributes $n(i)\in\mathcal{N},p(i)\in\mathcal{P}$ which indicate the associated node and product of the supplier. The supply bid cost is the cost at which the supplier offers product and is often related to its marginal supply cost. The supply bid capacity is the maximum flow of product that it can offer.  Note that the supply bid cost can be positive or negative. When the bid cost is positive (standard), the supplier expects the market to pay for its product. When the bid cost is negative, the supplier is willing to pay for the market to take away its product, this situation can arise when the product generates a financial loss (e.g., the supplier has excess product or generates a local environment impact). For instance, a city is a supplier that offers to pay the market (negative bid) for taking its waste. 

Each consumer $j\in\mathcal{D}$ has a physical demand variable $d_{j}\in\mathbb{R}_{+}$, a bid capacity $\overline{d}_{j}\in\mathbb{R}_{+}$, a bid cost denoted $\alpha^{d}_{j}\in\mathbb{R}$, and set attributes $n(j)\in\mathcal{N},p(j)\in\mathcal{P}$ indicating its location and desired product. The demand bid cost is the cost that consumer is willing to pay for a product and the capacity is the maximum flow that it can take. A positive bid cost (standard), indicates that the consumer is willing to pay to the market to obtain product. When the bid cost is negative, the consumer expects to be paid in order to take a product, this situation arises when taking product generates a financial loss for the consumer (e.g., an environmental impact). For instance, a landfill is a consumer that requests a payment from the market (negative bid) for taking waste. Consumers with negative bids can also be used to capture policy makers (that represent society of the environment). For instance, a consumer that represents a lake can charge the market if a product (e.g., wastewater) is discharged in it.   

We define the nested sets $\mathcal{S}_{n,p}\subseteq\mathcal{S}_{n}\subseteq\mathcal{S}$ where $S_{n}:=\{i|n(i)=n\}$ and $S_{n,p}:=\{i|n(i)=n,p(i)=p\}$. The set $S_{n}$ includes all suppliers $i$ located at a node $n$; the set $S_{n,p}$ refines inclusion to suppliers of a product $p$ at a node $n$. Following a similar logic, we define the nested sets $\mathcal{D}_{n,p}\subseteq\mathcal{D}_{n}\subseteq\mathcal{D}$ with $D_{n}:=\{j|n(j)=n\}$ and $D_{n,p}:=\{j|n(j)=n,p(j)=p\}$. This notation allows us to categorize suppliers and consumers by product and location. 

Transport providers $l\in\mathcal{L}$ each possesses a physical flow variable $f_{l}\in\mathbb{R}_{+}$, a bid capacity $\overline{f}_{l}\in\mathbb{R}_{+}$, a bid cost $\alpha^{f}_{l}\in\mathbb{R}_{+}$, and attributes $n_{s}(l)\in\mathcal{N}$ (its source node) $n_{r}(l)\in\mathcal{N}$ (its receiving node) and $p(l)\in\mathcal{P}$ representing the product it will move. The bid cost is the cost that the provider charges for its service; the bid cost is assumed here to be positive because it is rare to encounter situations in which it is negative (although mathematically this is possible). We will see that positive transport costs are necessary to prevent inefficient ``cycling" behavior (inefficient use of transport). We apply similar nesting to transport providers $\mathcal{L}$ and extend the sets to capture transport direction. We define $\mathcal{L}_{n,p}^{in}\subseteq\mathcal{L}_{n}^{in}\subseteq\mathcal{L}$ and $\mathcal{L}_{n,p}^{out}\subseteq\mathcal{L}_{n}^{out}\subseteq\mathcal{L}$ to differentiate inbound and outbound transport from a node, where $\mathcal{L}_{n}^{in}:=\{l|n_r(l)=n\}$ and $\mathcal{L}_{n}^{out}:=\{l|n_s(l)=n\}$, and $\mathcal{L}_{n,p}^{in}:=\{l|n_r(l)=n,p(l)=p\}$ and $\mathcal{L}_{n,p}^{out}:=\{l|n_s(l)=n,p(l)=p\}$.

Technology providers $t\in\mathcal{T}$ have physical variables $c_{t}\in\mathbb{R}_{+}$ and $g_{t}\in\mathbb{R}_{+}$,  representing product consumption and generation, respectively. Each $t\in\mathcal{T}$ also has bid capacities $\overline{c}_{t}\in\mathbb{R}_{+}$ and $\overline{g}_{t}\in\mathbb{R}_{+}$, and bid cost  $\alpha^{c}_{t}\in\mathbb{R}_{+}$. The processing bid cost is the cost that the provider charges for providing processing services (e.g., operating costs for a technology). We will see that negative bid costs can induce inefficient use of technologies.   Technology attributes include a node $n(t)\in\mathcal{N}$, a set of input products $\mathcal{P}_{t}^{con}\subseteq\mathcal{P}$ to the technology, and a set of output products $\mathcal{P}_{t}^{gen}\subseteq\mathcal{P}$ generated by the technology.  In other words, the technology can take multiple inputs to generate multiple outputs (e.g., a recycling facility or a waste-to-fuels facility). We use technology inputs and outputs to define the nested sets $\mathcal{T}_{n,p}^{con}\subseteq\mathcal{T}_{n}\subseteq\mathcal{T}$ and $\mathcal{T}_{n,p}^{gen}\subseteq\mathcal{T}_{n}\subseteq\mathcal{T}$ where $\mathcal{T}_{n}:=\{t|n(t)=n\}$, $\mathcal{T}_{n,p}^{con}:=\{t|n(t)=n,p(t)\in\mathcal{P}_{t}^{con}\}$ and $\mathcal{T}_{n,p}^{gen}:=\{t|n(t)=n,p(t)\in\mathcal{P}_{t}^{gen}\}$ such that technologies may be differentiated by their input and output product attributes. We select a product $\bar{p}\in\mathcal{P}_{t}^{con}$ as the {\em reference} product for technology $t$. Yield coefficients $\gamma_{t,p}\in\mathbb{R}_{+}, p\in\{\mathcal{P}_{t}^{con},\mathcal{P}_{t}^{gen}\}$ are defined with respect to this reference product, and $\gamma_{t,\bar{p}}=1$. This structuring of technology providers defines product consumption and generation rates relative to the reference product. Bids for processing  providers are defined with respect to this reference product; the notation $\alpha^{c}_{t}$ indicates that bids are made with respect to their consumption (processing) rates. We note that all supply, demand, transport, and processing flows are positive (because these are physical flows). 

\subsection{Market Clearing Problem}\label{sec: market primal}

The ISO uses stakeholder bidding information $(\alpha^{s},\alpha^{d},\alpha^{f},\alpha^{c})$ and capacities $(\overline{s},\overline{d},\overline{c},\overline{f})$ to solve an optimization problem that is known as the {\em market clearing} problem. The goal of the clearing problem is to allocate physical resources and associated economic values to the stakeholders in the system. It is important to emphasize that the mathematical formulation of the clearing problem defines how economic value is generated and distributed among the stakeholders and thus must be  designed with care. Deriving proper clearing formulations is the central topic of market design. 

The objective of the market clearing problem that we propose is given by \eqref{Obj}. The objective seeks to find physical allocations $(s,d,f,c,g)$ for all involved stakeholders that maximize the social welfare.  Fundamentally, maximizing the social welfare seeks to maximize the balance between value (positive terms) with costs (negative terms). If all the bids are positive, maximizing the social welfare seeks to maximize demand served to consumers while minimizing supply, transport, and processing costs. 
\begin{subequations}
\label{Primal}
\begin{equation}
\label{Obj}
\begin{aligned}
\max_{s,d,f,c,g} \quad \sum_{j \in \mathcal{D}}{\alpha_{j}^{d}}d_{j} - \sum_{i \in \mathcal{S}}{\alpha_{i}^{s}}s_{i} - \sum_{l\in\mathcal{L}}{\alpha^{f}_{l}f_{l}} - \sum_{t\in\mathcal{T}^{con}_{n,\bar{p}}}{\alpha^{c}_{t}c_{t}}
\end{aligned}
\end{equation}
We note that the social welfare indicates that priority will be given to consumers with large bid costs and suppliers and service providers with low bid costs. However, the allocations need to satisfy a number of constraints associated with physical conservation and capacity. To define such constraints, we define the total nodal and product supply and demand allocations on the SC graph as:
\begin{equation}
\label{SupRel}
\begin{aligned}
s_{n,p} = \sum_{i\in\mathcal{S}_{n,p}}{s_{i}}, \quad (n,p)\in\mathcal{N}\times\mathcal{P},\; (\pi^{sup}_{n,p})
\end{aligned}
\end{equation}
\begin{equation}
\label{DemRel}
\begin{aligned}
d_{n,p} = \sum_{j\in\mathcal{D}_{n,p}}{d_{j}}, \quad (n,p)\in\mathcal{N}\times\mathcal{P},\; (\pi^{dem}_{n,p})
\end{aligned}
\end{equation}
Here, the notation  $s_{i}$ and $d_{j}$ is is short-hand notation for $s_{n(i),p(i)}$ and $d_{n(j),p(j)}$ (mapping of stakeholders onto the SC graph). The variable $s_{n,p}\in\mathbb{R}_{+}$ is the total supply flow injected at a given node and for a given product. The variable $d_{n,p}\in\mathbb{R}_{+}$ is the total demand flow withdrawn from a given node and for a given product. These definitions capture the fact that multiple stakeholders $i\in\mathcal{S}$ or $j\in\mathcal{D}$ can bid for the same product or from the same location. We also define the dual variables corresponding to these constraints as $\pi^{sup}_{n,p}\in\mathbb{R}$ and $\pi^{dem}_{n,p}\in\mathbb{R}$.

Physical conservation within the model is provided by:
\begin{equation}
\label{Balance}
\begin{aligned}
s_{n,p} + \sum_{l\in\mathcal{L}^{in}_{n,p}}{f_{l}} +
 \sum_{t\in\mathcal{T}_{n,p}^{gen}}{g_{t}} = d_{n,p} + \sum_{l\in\mathcal{L}^{out}_{n,p}}{f_{l}} + \sum_{t\in\mathcal{T}_{n,p}^{con}}{c_{t}}, \quad (n,p)\in\mathcal{N}\times\mathcal{P},\; (\pi^{bal}_{n,p})
\end{aligned}
\end{equation}
In a market context, these conservation constraints are referred to as the {\em clearing constraints}. From the perspective of a SC, these are known as nodal product balances. These constraints ensure that incoming supply, transport flows, and generated flows for a given product and  node equal demand served, outgoing transportation, and product consumption. The dual variables of these constraints are defined as $\pi^{bal}_{n,p}\in\mathbb{R}$. In Section \ref{sec: properties} we will demonstrate that these dual variables represent nodal market clearing prices and play an important role in defining economic value for products at different locations and in determining how stakeholders are charged or remunerated.

Constraint \eqref{Conversion} defines product transformation (conversion) relationships for the processing technologies. Product transformations are expressed as linear equations with yield parameters $\gamma_{t,p}\in\mathbb{R}_{+}$ for each product involved. The dual variable to \eqref{Conversion} is $\pi^{con}_{t,n,p,p'}\in\mathbb{R}$.
\begin{equation}
\label{Conversion}
\begin{aligned}
\gamma_{t,p'}g_{t,n(t),p(t)} = \gamma_{t,p}c_{t,n(t),p'(t)}, \quad (t,n,p,p')\in\mathcal{T}\times\mathcal{N}\times\mathcal{P}_{t}^{gen}\times\mathcal{P}_{t}^{con},\; (\pi^{con}_{t,n,p,p'})
\end{aligned}
\end{equation}

Equations \eqref{Smax} to \eqref{gmax} provide bounds for stakeholder allocations (maximum flows they can provide or take). Nodal supply and demand capacities $\overline{s}_{n,p}\in \mathbb{R}_{+}$ and $\overline{d}_{n,p}\in \mathbb{R}_{+}$ are derived from $\overline{s}_{i}$ and $\overline{d}_{j}$ and are equivalent. The dual variables for each of the upper bounding constraints are defined as $\overline{\lambda}_{i}\in \mathbb{R}_{+}$, $\overline{\lambda}_{j}\in \mathbb{R}_{+}$, $\overline{\lambda}^{d}_{n,p}\in \mathbb{R}_{+}$, $\overline{\lambda}{s}_{n,p}\in \mathbb{R}_{+}$, $\overline{\lambda}_{l}\in \mathbb{R}_{+}$, $\overline{\lambda}^{c}_{t}\in \mathbb{R}_{+}$, and $\overline{\lambda}^{g}_{t}\in \mathbb{R}_{+}$.
\begin{equation}
\label{Smax}
\begin{aligned}
s_{i}\leq \overline{s}_{i},\quad i\in\mathcal{S},\; (\overline{\lambda}_{i})
\end{aligned}
\end{equation}
\begin{equation}
\label{Dmax}
\begin{aligned}
d_{j}\leq \overline{d}_{j},\quad j\in\mathcal{D},\; (\overline{\lambda}_{j})
\end{aligned}
\end{equation}
\begin{equation}
\label{dmax}
\begin{aligned}
d_{n,p}\leq \overline{d}_{n,p},\;(n,p)\in \mathcal{N}\times\mathcal{P},\; (\overline{\lambda}^{d}_{n,p})
\end{aligned}
\end{equation}
\begin{equation}
\label{smax}
\begin{aligned}
s_{n,p}\leq \overline{s}_{n,p},\;(n,p)\in \mathcal{N}\times\mathcal{P},\; (\overline{\lambda}^{s}_{n,p})
\end{aligned}
\end{equation}
\begin{equation}
\label{fmax}
\begin{aligned}
f_{l}\leq \overline{f}_{l},\;(l)\in \mathcal{L},\; (\overline{\lambda}_{l})
\end{aligned}
\end{equation}
\begin{equation}
\label{xmax}
\begin{aligned}
c_{t}\leq \overline{c}_{t},\;(t)\in \mathcal{T},\; (\overline{\lambda}^{c}_{t})
\end{aligned}
\end{equation}
\begin{equation}
\label{gmax}
\begin{aligned}
g_{t}\leq \overline{g}_{t},\;(t)\in \mathcal{T},\; (\overline{\lambda}^{g}_{t})
\end{aligned}
\end{equation}
\end{subequations}

In summary, the market clearing model is given by  the set of equations \eqref{Primal} and seeks allocations that maximize the social welfare and satisfy conservation and capacity constraints. We also highlight the presence of dual variables (also known as hidden or shadow variables); in typical SC settings these variables are ignored, but in a market interpretation these play a key role in establishing economic values for products and services and dictate stakeholder remuneration. 

\subsection{Economic Properties of Bids, Prices, and Profits}\label{sec: properties}

In this section we will see that bidding information provided to the ISO (bid costs and capacities) define the feasible space for market clearing allocations $(s,d,f,c,g)$ and prices $\pi^{bal}_{n,p}$. This is of interest because clearing prices and physical allocations represent charges to consumers and remuneration for suppliers, transporters, and technology providers. 

The dual variables $\pi^{bal}_{n,p}$ are indexed by SC node $n\in\mathcal{N}$ and product $p\in\mathcal{P}$ (they represent the value of a product $p$ at a node $n$); accordingly, they are also referred to as nodal prices. These nodal prices represent the economic value of products  at specific geographical locations.  The following relationships define short-hand notation for prices (as done for allocations):
\begin{subequations}\label{NodalPrices}
\begin{equation}
\label{pi_i}
\begin{aligned}
\pi_{i} := \pi^{bal}_{n(i),p(i)},\; i\in\mathcal{S}
\end{aligned}
\end{equation}
\begin{equation}
\label{pi_j}
\begin{aligned}
\pi_{j} := \pi^{bal}_{n(j),p(j)},\; j\in\mathcal{D}
\end{aligned}
\end{equation}
Geographical differences between nodal prices capture the economic value of product transportation and these are defined as:
\begin{equation}
\label{pi_f}
\begin{aligned}
\pi_{l} := \pi^{bal}_{n'(l),p(l)} - \pi^{bal}_{n(l),p(l)},\; l\in\mathcal{L}
\end{aligned}
\end{equation}
We define the economic value of product transformation, which is given by the difference between the prices of products consumed and generated by a technology (weighted by yield coefficients $\gamma_{t,p}$):
\begin{equation}
\label{pi_t}
\begin{aligned}
\pi_{t} := \sum_{p\in\mathcal{P}_{t}^{gen}}{\gamma_{t,p}\pi^{bal}_{n(t),p}} - \sum_{p'\in\mathcal{P}_{t}^{con}}{\gamma_{t,p'}\pi^{bal}_{n(t),p'}},\; t\in\mathcal{T}
\end{aligned}
\end{equation}
\end{subequations}
We denote the entire set of market clearing prices $\pi=(\pi_{i},\pi_{j},\pi_{l},\pi_{t})$; this shorthand notation allows us to associate economic  values and profits with stakeholders in an intuitive way. 

In the proposed setting,  $\alpha^{d}_{j}d_{j}$ is the economic value of its bid allocation and $\pi_{j}d_{j}$ is its remuneration from the market. The consumer thus has an incentive to participate in the market (buy product) if $\alpha^{d}_{j}>\pi_{j}$ (the value provided by the product is higher than what it pays for). For suppliers we have that $\alpha^{s}_{i}s_{i}$ is the value of the supplied  product  (interpreted as a marginal or operating cost) and receive $\pi_{i}s_{i}$ in revenue from the market. Suppliers are thus motivated to sell product to the market if $\pi_{i}>\alpha^{s}_{i}$. Using a similar logic, transportation and technology providers have operating costs of $\alpha^{f}_{l}f_{l}$ and $\alpha^{c}_{t}c_{t}$ and they receive $\pi_{l}f_{l}$ and $\pi_{t}c_{t}$ in revenue accordingly from the market. Specifically, these players are motivated to provide services to the market if $\pi_{l}>\alpha^{f}_{l}$ and $\pi_{t}>\alpha^{c}_{t}$. Following this logic, we define the stakeholder profits as:
\begin{subequations}\label{Profits}
\begin{equation}
\label{Phi_s}
\begin{aligned}
\phi^{s}_{i}(\pi_{i},\alpha^{s}_{i},s_{i}) = (\pi_{i} - \alpha^{s}_{i})s_{i},\quad i\in\mathcal{S}
\end{aligned}
\end{equation}
\begin{equation}
\label{Phi_d}
\begin{aligned}
\phi^{d}_{j}(\pi_{j},\alpha^{d}_{j},d_{j}) = (\alpha^{d}_{j} - \pi_{j})d_{j},\quad j\in\mathcal{D}
\end{aligned}
\end{equation}
\begin{equation}
\label{Phi_f}
\begin{aligned}
\phi^{f}_{l}(\pi_{l},\alpha^{f}_{l},f_{l}) = (\pi_{l} - \alpha^{f}_{l})f_{l},\quad l\in\mathcal{L}
\end{aligned}
\end{equation}
\begin{equation}
\label{Phi_t}
\begin{aligned}
\phi^{c}_{t}(\pi_{t},\alpha^{c}_{t},c_{t}) = (\pi_{t} - \alpha^{c}_{t})c_{t},\quad t\in\mathcal{T}
\end{aligned}
\end{equation}
\end{subequations}
Intuitively, players have an incentive to participate in the market if their profits are positive, are ambivalent if they are zero, and do not have an incentive if they are negative. Accordingly, we want to ensure that the market clearing formulation delivers allocations and clearing prices that lead to non-negative profits (otherwise, the market does not provide a viable environment for the stakeholders). We use $\phi=(\phi^{s}_{i},\phi^{d}_{j},\phi^{f}_{l},\phi^{c}_{t})$ as  shorthand notation for all profits. 

We now show that Lagrangian dual of the market clearing model indicates that pricing and profit definitions emerge naturally. The Lagrangian dual problem is defined as:
\begin{equation}
\begin{aligned}
\max_{\pi}\min_{s,d,f,c,g}{\mathcal{L}(s,d,f,c,g,\pi)}.
\end{aligned}
\end{equation}
The Lagrange function is constructed in \eqref{CMLagrangian}, and contains the supply and demand geographical mapping constraints, the market clearing constraint, and the conversion constraint which are all raised into the Lagrangian function with their appropriate dual variables.
\begin{multline}
\label{CMLagrangian}
\mathcal{L}(s,d,f,c,g,\pi)=-\sum_{j \in \mathcal{D}}{\alpha_{j}^{d}}d_{j} + \sum_{i \in \mathcal{S}}{\alpha_{i}^{s}}s_{i} + \sum_{l\in\mathcal{L}}{\alpha^{f}_{l}f_{l}} + \sum_{t\in\mathcal{T}^{con}_{n,p}}{\alpha^{c}_{t}c_{t}}\\
+\sum_{n\in\mathcal{N}}\sum_{p\in\mathcal{P}}\pi^{sup}_{n,p}\left(\sum_{i\in\mathcal{S}_{n,p}}{s_{i}} - s_{n,p}\right)
+  \sum_{n\in\mathcal{N}}\sum_{p\in\mathcal{P}}\pi^{dem}_{n,p}\left(d_{n,p} - \sum_{j\in\mathcal{D}_{n,p}}{d_{j}}\right)\\
+ \sum_{n\in\mathcal{N}}\sum_{p\in\mathcal{P}}\pi^{bal}_{n,p}\left(d_{n,p} + \sum_{l\in\mathcal{L}^{out}_{n,p}}{f_{l}} + \sum_{t\in\mathcal{T}_{n,p}^{con}}{c_{t}} - s_{n,p} - \sum_{l\in\mathcal{L}^{in}_{n,p}}{f_{l}} -
\sum_{t\in\mathcal{T}_{n,p}^{gen}}{g_{t}}\right)\\
+\sum_{t\in\mathcal{T}}\sum_{n\in\mathcal{N}}\sum_{p'\in\mathcal{P}^{con}_{t}}\sum_{p\in\mathcal{P}^{gen}_{t}}{\pi^{con}_{t,n,p,p'}\left(\gamma_{t,p}c_{t,n(t),p'(t)}  -\gamma_{t,p'}g_{t,n(t),p(t)}\right)}
\end{multline}
The market clearing formulation is linear and thus satisfies strong duality \cite{bazaraa2011linear}. Strong duality indicates that a solution of the Lagrangian dual problem is also a solution of the primal formulation).

The Lagrange function can be simplified by using the following identities:
\begin{subequations}\label{Identities}
\begin{equation}
\begin{aligned}
\sum_{n\in\mathcal{N}}\sum_{p\in\mathcal{P}}\pi^{bal}_{n,p}\sum_{i\in\mathcal{S}_{n,p}}s_{i}=\sum_{i\in\mathcal{S}}\pi_{i}s_{i},
\end{aligned}
\end{equation}
\begin{equation}
\begin{aligned}
\sum_{n\in\mathcal{N}}\sum_{p\in\mathcal{P}}\pi^{bal}_{n,p}\sum_{j\in\mathcal{D}_{n,p}}d_{j}=\sum_{j\in\mathcal{D}}\pi_{j}d_{j},
\end{aligned}
\end{equation}
\begin{equation}
\begin{aligned}
\sum_{n\in\mathcal{N}}\sum_{p\in\mathcal{P}}\pi^{bal}_{n,p}\left(\sum_{l\in\mathcal{L}^{in}_{n,p}}f_{l}-\sum_{l\in\mathcal{L}^{out}_{n,p}}f_{l}\right)&=\sum_{l\in\mathcal{L}}\left(\pi^{bal}_{n'(l),p(l)}-\pi^{bal}_{n(l),p(l)}\right)f_{n(l),n'(l),p(l)}\\
&=\sum_{l\in\mathcal{L}}\pi_{l}f_{l}
\end{aligned}
\end{equation}
and,
\begin{equation}
\begin{aligned}
\sum_{n\in\mathcal{N}}\sum_{p\in\mathcal{P}}\pi^{bal}_{n,p}\left(\sum_{t\in\mathcal{T}^{gen}_{n,p}}g_{t,n,p}-\sum_{t\in\mathcal{T}^{con}_{n,p}}c_{t,n,p}\right)&=\sum_{n\in\mathcal{N}}\sum_{p\in\mathcal{P}}\pi^{bal}_{n,p}\left(\sum_{t\in\mathcal{T}^{gen}_{n,p}}\gamma_{t,p}c_{t,n,\bar{p}} - \sum_{t\in\mathcal{T}^{con}_{n,p}}\gamma_{t,p}c_{t,n,\bar{p}}\right)\\
&=\sum_{t\in\mathcal{T}}\pi_{t}c_{t,n(t),\bar{p}(t)}
\end{aligned}
\end{equation}
\end{subequations}
The technology price identity makes use of yield factor definitions to rewrite the expression solely in terms of the reference product $\bar{p}(t)$.

We observe that terms associated with the dual variables $\pi^{sup}_{n,p}$, $\pi^{dem}_{n,p}$, and $\pi^{con}_{t,n,p,p'}$ each add to zero and thus cancel out of the Lagrange function. Applying the price identities to the remaining terms leads to \eqref{FinalLagrangian}. This analysis allows us to express the Lagrangian solely in terms of stakeholder profits and indicates that the Lagrangian dual minimizes negative profits (maximizes profits). 
\begin{equation}
\label{FinalLagrangian}
\begin{aligned}
\mathcal{L}(s,d,f,c,g,\pi)=&-\sum_{j \in \mathcal{D}}{\alpha_{j}^{d}}d_{j} + \sum_{i \in \mathcal{S}}{\alpha_{i}^{s}}s_{i} + \sum_{l\in\mathcal{L}}{\alpha^{f}_{l}f_{l}} + \sum_{t\in\mathcal{T}}{\alpha^{c}_{t}c_{t,n(t),\bar{p}(t)}}\\
&\qquad +\sum_{j\in\mathcal{D}}\pi_{j}d_{j}-\sum_{i\in\mathcal{S}}\pi_{i}s_{i}-\sum_{l\in\mathcal{L}}\pi_{l}f_{l}-\sum_{t\in\mathcal{T}}\pi_{t}c_{t,n(t),\bar{p}(t)}\\
=&-\left(\sum_{j \in \mathcal{D}}{\phi_{j}^{d}} + \sum_{i \in \mathcal{S}}{\phi_{i}^{s}} + \sum_{l\in\mathcal{L}}{\phi^{f}_{l}} + \sum_{t\in\mathcal{T}}{\phi^{c}_{t}}\right)
\end{aligned}
\end{equation}

\subsection{Dual of Market Clearing Problem}\label{sec: market dual}

The dual formulation to the clearing model provides additional insight (not obvious from the Lagrangian dual) on economic properties that emerge from the clearing procedure. We first note that it is possible to condense the primal clearing formulation \eqref{Primal} by incorporating the conversion constraints directly into the market clearing constraints, and writing the model exclusively in terms of stakeholder-indexed supply and demand volumes. The condensed model also replaces the variables $c_{t}$ and $g_{t}$ with a single technology flow variable $\xi_{t}\in\mathbb{R}_{+}$ with the same attributes. The condensed primal clearing formulation is given by  \eqref{PrimalCondensed}.
\begin{subequations}\label{PrimalCondensed}
\begin{equation}
\label{ObjCondensed}
\begin{aligned}
\max_{s,d,f,\xi} \quad \sum_{j \in \mathcal{D}}{\alpha_{j}^{d}}d_{j} - \sum_{i \in \mathcal{S}}{\alpha_{i}^{s}}s_{i} - \sum_{l\in\mathcal{L}}{\alpha^{f}_{l}f_{l}} - \sum_{t\in\mathcal{T}}{\alpha^{\xi}_{t}\xi_{t}}
\end{aligned}
\end{equation}
\begin{multline}
\label{BalanceCondensed}
\textrm{s.t.}\; \sum_{i\in\mathcal{S}_{n,p}}{s_{i}} + \sum_{l\in\mathcal{L}^{in}_{n}}{f_{l}} +
\sum_{t\in\mathcal{T}_{n,p}^{gen}}{\gamma_{t,p}\xi_{t}} =
\sum_{j\in\mathcal{D}_{n,p}}{d_{j}} + \sum_{l\in\mathcal{L}^{out}_{n}}{f_{l}} + \sum_{t\in\mathcal{T}_{n,p}^{con}}{\gamma_{t,p}\xi_{t}}, \quad (n,p)\in\mathcal{N}\times\mathcal{P},\; (\pi^{bal}_{n,p})
\end{multline}
\begin{equation}
\label{SmaxCondensed}
\begin{aligned}
s_{i}\leq \overline{s}_{i},\quad i\in\mathcal{S},\; (\overline{\lambda}_{i})
\end{aligned}
\end{equation}
\begin{equation}
\label{DmaxCondensed}
\begin{aligned}
d_{j}\leq \overline{d}_{j},\quad j\in\mathcal{D},\; (\overline{\lambda}_{j})
\end{aligned}
\end{equation}
\begin{equation}
\label{fmaxCondensed}
\begin{aligned}
f_{l}\leq \overline{f}_{l},\quad l\in\mathcal{L},\; (\overline{\lambda}_{l})
\end{aligned}
\end{equation}
\begin{equation}
\label{xmaxCondensed}
\begin{aligned}
\xi_{t}\leq \overline{\xi}_{t},\quad t\in \mathcal{T},\; (\overline{\lambda}_{t})
\end{aligned}
\end{equation}
\end{subequations}
The corresponding condensed dual of the condensed formulation is given by \eqref{LPDualCondensed}. The dual objective is presented in \eqref{dualObjCondensed}; this function contains only the dual variables corresponding to the capacity constraints of \eqref{Primal}. This provides some intuition regarding dual behavior; the dual objective minimizes capacity constraint margins. In other words, it seeks to extract as much capacity from the players as possible. How these margins relate to market prices and bid costs is defined by the dual constraints.
\begin{subequations}\label{LPDualCondensed}
\begin{equation}
\label{dualObjCondensed}
\begin{aligned}
\min_{\pi,\overline{\lambda}}
\sum_{i \in \mathcal{S}}{\overline{s}_{i}\overline{\lambda}_{i}} +
\sum_{j \in \mathcal{D}}{\overline{d}_{j}\overline{\lambda}_{j}} +
\sum_{l\in\mathcal{L}}{\overline{f}_{l}\overline{\lambda}_{l}} +
\sum_{t\in\mathcal{T}}{\overline{\xi}_{t}\overline{\lambda}_{t}}
\end{aligned}
\end{equation}
\begin{equation}
\label{dualCon_sCondensed}
\begin{aligned}
\pi^{bal}_{n(i),p(i)} - \overline{\lambda}_{i} \geq \alpha_{i}^{s},\; i\in\mathcal{S}
\end{aligned}
\end{equation}
\begin{equation}
\label{dualCon_dCondensed}
\begin{aligned}
\pi^{bal}_{n(j),p(j)} + \overline{\lambda}_{j} \leq \alpha_{j}^{d},\; j\in\mathcal{D}
\end{aligned}
\end{equation}
\begin{equation}
\label{dualCon_fCondensed}
\begin{aligned}
\pi^{bal}_{n(l),p(l)} - \pi^{bal}_{n'(l),p(l)} - \overline{\lambda}_{n(l),n'(l),p(l)} \geq \alpha^{f}_{l},\; l\in\mathcal{L}
\end{aligned}
\end{equation}
\begin{equation}
\label{dualCon_xCondensed}
\begin{aligned}
\sum_{p\in\mathcal{P}^{gen}_{t}}{\gamma_{t,p}\pi^{bal}_{n(t),p(t)}} - \sum_{p\in\mathcal{P}^{con}_{t}}{\gamma_{t,p}\pi^{bal}_{n(t),p(t)}} - \overline{\lambda}_{t} \geq \alpha^{\xi}_{t},\; t\in\mathcal{T}
\end{aligned}
\end{equation}
\end{subequations}
The dual problem reveals the same important pricing properties as determined using the Lagrangian dual. Rewriting the dual constraints in the compact form \eqref{DualConstraintPrices} (with $\pi_{i}$ and $\pi_{j}$ representing nodal prices for suppliers and consumers, $\pi_{l}$ for transportation prices, and $\pi_{t}$ for technology prices) gives:
\begin{subequations}\label{DualConstraintPrices}
\begin{equation}
\label{dualCon_sIdent}
\begin{aligned}
\pi_{i} - \overline{\lambda}_{i} \geq \alpha_{i}^{s},\; i\in\mathcal{S}
\end{aligned}
\end{equation}
\begin{equation}
\label{dualCon_dIdent}
\begin{aligned}
\pi_{j} + \overline{\lambda}_{j} \leq \alpha_{j}^{d},\; j\in\mathcal{D}
\end{aligned}
\end{equation}
\begin{equation}
\label{dualCon_fIdent}
\begin{aligned}
\pi_{l} - \overline{\lambda}_{l} \geq \alpha^{f}_{l},\; l\in\mathcal{L}
\end{aligned}
\end{equation}
\begin{equation}
\label{dualCon_xIdent}
\begin{aligned}
\pi_{t} - \overline{\lambda}_{t} \geq \alpha^{\xi}_{t},\; t\in\mathcal{T}
\end{aligned}
\end{equation}
\end{subequations}
The dual thus captures price relationships to bids and indicate that market prices are not allowed to exceed consumer bids nor fall below supplier, transporter, and technology bids (recall that the duals $\overline\lambda$ are non-negative). The presence of $\overline\lambda$ in each relationship implies that price values are related to whether or not market demands are fully served under the interpretation of $\overline\lambda$ as a marginal price. Importantly, the arrangement of these $\overline\lambda$ is such that the bid bounding properties cannot be violated (i.e., the duals serve to tighten these constraints).

Note that \eqref{dualCon_sIdent} can also be written as:
\begin{equation*}
\pi_{i} -\alpha_{i}^{s}  \geq \overline{\lambda}_{i},\; i\in\mathcal{S}
\end{equation*}
This reveals that $\pi_{i} -\alpha_{i}^{s}\geq 0$ ($\pi_{i}\geq \alpha_{i}^{s}$); moreover, this indicates minimizing the dual objective (minimize  $\overline{\lambda}_{i}$), seeks to widen the gap between the clearing price $\pi_{i}$ and the supplier bid $\alpha_{i}^{s}$ (and with this remunerate the stakeholder more). Similar logic can be followed for the other stakeholders. 

The dual constraints can be interpreted as {\em economic driving forces}. In each of the constraints, a price value $\pi_{i}$, $\pi_{j}$, $\pi_{l}$, or $\pi_{t}$ higher than (for suppliers, transportation providers, and technology providers) or lower than (for consumers) the associated bid value represents a positive profit associated with the transaction and provides a driving force for that stakeholder to participate in the market. The price identities help provide some context for these driving forces; a positive transport price $\pi_{l}$ with $n(l)$, $n'(l)$, and $p(l)$ implies an economic driving force to move product $p$ from node $n$ to node $n'$. A positive technology price $\pi_{t}$ with product outputs $\mathcal{P}^{gen}_{t}$ and product inputs $\mathcal{P}^{con}_{t}$ implies an economic driving force to transform the set of products $\mathcal{P}^{con}_{t}$ into the set of products $\mathcal{P}^{gen}_{t}$ because the generated products are worth more than the products consumed (accounting for yield coefficients). Moreover, these relationships imply that negative profits are not allowed in coordinated markets. In other words, the clearing procedure delivers allocations such that no stakeholder loses money. 

The relationships between the primal and dual formulations are visualized in Figure~\ref{fig_PrimalDual}. The primal formulation of the ISO market clearing problem provides a representation explicitly in terms of the physical allocations $(s,d,f,\xi)$ while the dual representation is explicit in economic allocations $\pi$. Using either model formulation, the other half of the solution is obtained from constraint shadow prices. The correspondence between each primal and dual variable to a constraint in the other formulation is reiterated in Figure~\ref{fig_PrimalDual}.
\begin{figure}[!htb]
	\center{\includegraphics[width=1.0\textwidth]{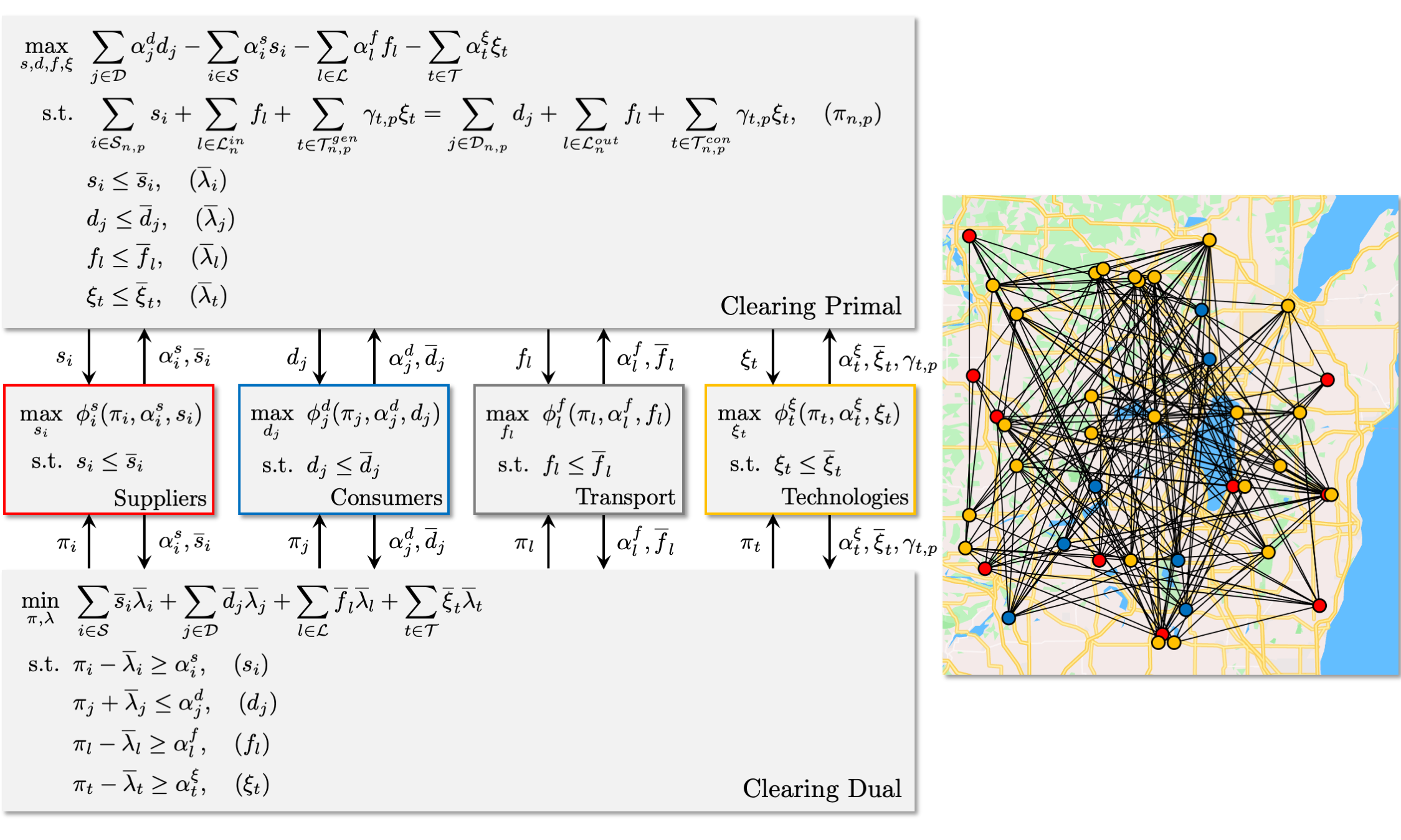}}
	\caption{Complementary primal and dual views of coordinated market. Suppliers, consumers, transport providers, and technology providers submit bidding information to the ISO, who solves the clearing problem. The clearing problem maximizes the social welfare (profit) of all stakeholders by balancing supply and demand, finding efficient transportation routes, and managing product transformations. The primal view of the clearing problem defines physical allocations while the dual view defines economic allocations. Primal-dual variable correspondence is indicated for each constraint.}
	\label{fig_PrimalDual}
\end{figure}

\subsection{Theoretical Properties of Coordinated Markets}

We now exploit the primal-dual properties of the clearing formulation to define fundamental economic properties for the market. Each of these has been established in Sampat \textit{et. al.}~\cite{Sampat2018B} and we sketch the proofs here for completeness. These theorems are provided as a basis for further analysis, and we thus provide context around each.
\begin{theorem}
	\label{ThmProfit}
	The market delivers prices $\pi$ and allocations $(s,d,f,c,g)$ that maximize the collective profits $(\phi^{s},\phi^{d},\phi^{f},\phi^{c})$ and the profits are all nonnegative.
\end{theorem}
\begin{proof}
	For some arbitrary and fixed set of prices $\pi$ the allocation $(s,d,f,c,g)=(0,0,0,0,0)$ results in zero profits $\phi=(0,0,0,0)$ for all stakeholders. The Lagrangian \eqref{FinalLagrangian} is the total of negative profits, and $\mathcal{L}(0,0,0,0,0,\pi)=0$. Thus, subject to fixed $\pi$, solving the Lagrangian dual problem produces an allocation $(s,d,f,c,g)$ that minimizes the Lagrangian function, meaning $\mathcal{L}(s,d,f,c,g)\leq0$. Under fixed $\pi$, solving the Lagrangian is equivalent to maximizing the sum of the stakeholder profits, i.e., maximizing stakeholder profits independently, and therefore the allocations $(s,d,f,c,g)$ are at least as good as $(0,0,0,0,0)$. It follows that $\phi^{s}_{i}(\pi_{i},s_{i})\geq0$, $\phi^{d}_{j}(\pi_{j},d_{j})\geq0$, $\phi^{f}_{l}(\pi_{l},f_{l})\geq0$, and $\phi^{c}_{t}(\pi_{t},c_{t})\geq0$. Since $\pi$ are arbitrary prices, we have that the profits are non-negative for optimal clearing prices and the corresponding allocations.
\end{proof}

Stakeholder profits will be either non-negative if a stakeholder participates in the market, or zero if it does not. The interpretation here is an important one: formulating a SC as a coordinated market is sufficient to ensure that every stakeholder is sufficiently remunerated, else some stakeholders refuse to participate and some part of the market becomes dry.
\begin{theorem}
	\label{ThmComp}
	The market delivers prices $\pi$ and allocations $(s,d,f,c,g)$ that represent a competitive equilibrium.
\end{theorem}
\begin{proof}
	Let us choose an arbitrary set of prices $\pi$; minimizing the Lagrange function subject to these fixed $\pi$ produces allocations $(s,d,f,c,g)$ which maximize stakeholder profits independently. The market clearing program is linear, therefore strong duality holds and the allocations produced by solving the Lagrangian also solve the primal market clearing problem and satisfies the associated market clearing constraints.
\end{proof}

The above property indicates that coordination does not imply lack of competition and that the coordination enabled by the bidding  process to the ISO does not interfere with the competitive nature of the market. 
\begin{theorem}
	\label{RevAqc}
	The market delivers prices $\pi$ and allocations $(s,d,f,x)$ that lead to revenue adequacy:
	\begin{equation*}
	\begin{aligned}
	\sum_{j\in\mathcal{D}}{\pi_{j}d_{j}}=\sum_{i\in\mathcal{S}}{\pi_{i}s_{i}}+\sum_{l\in\mathcal{L}}{\pi_{l}f_{l}}+\sum_{t\in\mathcal{T}}{\pi_{t}c_{t}}
	\end{aligned}
	\end{equation*}
\end{theorem}
\begin{proof}
	Consider an optimal set of allocations $(s,d,f,c,g)$ for the market clearing problem; the clearing constraints hold at optimality, and we have from \eqref{Balance}:
	\begin{equation*}
	\begin{aligned}
	s_{n,p} + \sum_{l\in\mathcal{L}^{in}_{n}}{f_{l}} +
	\sum_{t\in\mathcal{T}_{n,p}^{gen}}{g_{t}} - d_{n,p} - \sum_{l\in\mathcal{L}^{out}_{n}}{f_{l}} - \sum_{t\in\mathcal{T}_{n,p}^{con}}{c_{t}}=0,
	\end{aligned}
	\end{equation*}
	Combining this result with the identities defining market prices in \eqref{Identities}, we express $\pi^{bal}_{n,p}$ in terms of prices $\pi_{i}$, $\pi_{j}$, $\pi_{l}$, and $\pi_{t}$, and also rewrite technology inputs and outputs in terms of reference products $\bar{p}(t)$. Altogether, this gives:
	\begin{equation*}
	\begin{aligned}
	\sum_{i\in\mathcal{S}}{\pi_{i}s_{i}}+\sum_{l\in\mathcal{L}}{\pi_{l}f_{l}}+\sum_{t\in\mathcal{T}}{\pi_{t}c_{t}}-\sum_{j\in\mathcal{D}}{\pi_{j}d_{j}}=0
	\end{aligned}
	\end{equation*}
	This recapitulates the revenue adequacy definition; it follows that revenue collected from consumers is sufficient to make payments to stakeholders.
\end{proof}

Revenue adequacy implies that financial payments provided by consumers are sufficient to cover the financial payments made to suppliers, transportation providers, and technology providers (there are no money losses in the market).
\begin{theorem}
	\label{ThmPrice}
	The market clearing prices satisfy the bounds: $\pi_{i}\geq\alpha^{S}_{i}$ for all $i\in\mathcal{S^*}$, $\pi_{j}\leq\alpha^{D}_{j}$ for all $j\in\mathcal{D^*}$, $\pi_{l}\geq\alpha^{f}_{l}$ for all $l\in\mathcal{L^*}$, and $\pi_{t}\geq\alpha^{c}_{t}$ for all $t\in\mathcal{T^*}$, where $\mathcal{S^*}$, $\mathcal{D^*}$, $\mathcal{T^*}$, and $\mathcal{L^*}$ indicate cleared transactions.
\end{theorem}
\begin{proof}
	Theorem~\ref{ThmProfit} establishes that profits $(\phi^{s},\phi^{d},\phi^{f},\phi^{c})$ are nonnegative. The allocations $(s,d,f,c,g)$ for all stakeholders are nonnegative as well, and strictly positive for cleared stakeholders. This implies that $\pi_{i}-\alpha^{s}_{i}\geq0$ for all $i\in\mathcal{S}^{*}$, $\pi_{j}-\alpha^{d}_{j}\geq0$ for all $j\in\mathcal{D}^{*}$, $\pi_{l}-\alpha^{f}_{l}\geq0$ for all $l\in\mathcal{L}^{*}$, and $\pi_{t}-\alpha^{c}_{t}\geq0$ for all $t\in\mathcal{T}^{*}$.
\end{proof}
We can thus see that the price bounding properties emerge from both dual and Lagrangian dual analysis, and have provided interpretations of price bounding properties as economic driving forces. We summarize by noting that these driving forces capture information about spatial and product value differentials and embed this information into the ISO clearing problem. Revenue adequacy also implies price behavior within the coordinated market model relative to stakeholder bid costs (because prices are bounded by bid costs). Specifically, revenue adequacy provides means of obtaining the bounding properties inferred from the dual program.

We observe that the market clearing framework is consistent under negative bid prices. We are interested in the specific behavior of a couple of such instances: negative consumer bids $\alpha^{d}_{j}<0$ and negative supplier bids $\alpha^{s}_{i}<0$. By the price bounding properties in Theorem~\ref{ThmPrice}, such bids allow the existence of corresponding prices $\pi_{j}<0$ and $\pi_{i}<0$. The physical interpretation of a consumer $j$ offering $\alpha^{d}_{j}<0$ is that it is requesting (rather than offering) payment to receive a product. Likewise, a supplier $i$ offering $\alpha^{s}_{i}<0$ offers payment to have a product removed. The ISO may thus collect revenue from both consumers and suppliers under appropriate bid values. We define the sets $\mathcal{S}^{-}:=\{i\in\mathcal{S}|\alpha^{s}_{i}<0\}$ and $\mathcal{S}^{+}:=\{i\in\mathcal{S}|\alpha^{s}_{i}\geq0\}$ for suppliers and $\mathcal{D}^{-}:=\{j\in\mathcal{D}|\alpha^{d}_{j}<0\}$ and $\mathcal{D}^{+}:=\{j\in\mathcal{D}|\alpha^{d}_{j}\geq0\}$ for consumers. Under this notation, the ISO can collect revenue from stakeholders $i\in\mathcal{S}^{-}$ and $j\in\mathcal{D}^{+}$. Using this bounding information to inform revenue collection in markets we can expand the definition of revenue adequacy.
\begin{corollary}
	The ISO can collect revenue from stakeholders $i\in\mathcal{S}^{-}$ and $j\in\mathcal{D}^{+}$ to achieve revenue adequacy.
\end{corollary}
\begin{proof}
	Follows from the proof of Theorem~\ref{RevAqc} with suppliers and consumers grouped into the subsets $\mathcal{S}^{+}$, $\mathcal{S}^{-}$, $\mathcal{D}^{+}$, and $\mathcal{D}^{-}$:
	\begin{equation*}
	\begin{aligned}
	\sum_{i\in\mathcal{S}^{+}}{\pi_{i}s_{i}}+\sum_{i\in\mathcal{S}^{-}}{\pi_{i}s_{i}}+\sum_{l\in\mathcal{L}}{\pi_{l}f_{l}}+\sum_{t\in\mathcal{T}}{\pi_{t}c_{t}}-\sum_{j\in\mathcal{D}^{+}}{\pi_{j}d_{j}}-\sum_{j\in\mathcal{D}^{-}}{\pi_{j}d_{j}}=0
	\end{aligned}
	\end{equation*}
\end{proof}

Another important (non-obvious) outcome of the clearing procedure is that it will not induce inefficient transportation cycles (as long as the transport bid costs are non-negative).
\begin{theorem}
	\label{ThmNoCycles}
	If the transportation provider bids satisfy $\alpha^{f}_{l}>0$ for all $l\in\mathcal{L}$ then no cleared allocation contains a transportation cycle.
\end{theorem}
\begin{proof}
	Consider a transportation cycle consisting of $c$ nodes in which a product $p$ is moved is moved from node to node in the cycle from $n_{1}\rightarrow n_{2}\rightarrow...\rightarrow n_{c-1}\rightarrow n_{c}\rightarrow n_{1}$. Theorem~\ref{ThmPrice} requires that $\pi_{l}\geq\alpha^{f}_{l}>0$ for the cleared transportation providers $l\in\mathcal{L}^{*}$. Following from this requirement, we have the pricing relationship around the cycle $\pi_{n_{c},p}>\pi_{n_{c-1},p}>...>\pi_{n_{2},p}>\pi_{n_{1},p}>\pi_{n_{c},p}$. The pricing requirements upon the cycle represent a logical contradiction.
\end{proof}

This property provides efficiency guarantees about the physical behavior of the allocations made by the ISO. Coordinated markets capture complex interdependencies between stakeholders and having this information available to the ISO avoids inefficient outcomes. In uncoordinated settings, particularly where transactions are conducted between individual stakeholders,  no such guarantees of efficiency exist. The result makes market coordination especially attractive where transportation costs represent a significant portion of total market value (e.g., waste collection).

\section{Enforcing Stakeholder Participation in Markets}\label{sec: forced markets}

Coordinated markets are based on voluntary stakeholder participation driven by natural economic incentives. The ISO bases its allocations on these incentives to deliver competitive results that generate maximum profit for stakeholders. It is important that the ISO does not assume the participation of any stakeholder when creating its allocations. In other words, from a policy standpoint, the ISO should be unable to force a stakeholder to participate in a market. From a modeling perspective, this means stakeholder allocations cannot be forced to specific values; doing so would imply that the ISO can compel stakeholder participation in the market. On the other hand, there are situations in which there are no incentives for a market to exist (e.g., waste markets) and thus one must determine appropriate mechanisms to activate the market.  Intuitively, this can be done by forcing participation (e.g., force provision of waste management services). In this section we show that forcing stakeholder participation destroys economic properties of a coordinated market. These results are important because traditional SC formulations enforce provision of services (e.g., satisfy demands or force supply). While intuitive from a primal perspective, such constructs can generate financially inefficient outcomes that alter stakeholder incentive). We use these results to argue that promoting participation in SCs should be attained by using economic incentives that preserve the properties of the market.  In the next section, we provide procedures to determine suitable incentives that promote participation (that activate markets). 

We define constraints of the type in \eqref{ForcingConstraints}, where $\underline{s}_{i}\in\mathbb{R}_{+}$, $\underline{d}_{j}\in\mathbb{R}_{+}$, $\underline{f}_{l}\in\mathbb{R}_{+}$,  and $\underline{c}_{t}\in\mathbb{R}_{+}$ represent minimum enforced allocations assumed by the ISO. For each enforcing participation constraint, we also define a dual variable $\underline{\lambda}_{i}\in\mathbb{R}_{+}$, $\underline{\lambda}_{j}\in\mathbb{R}_{+}$, $\underline{\lambda}_{l}\in\mathbb{R}_{+}$,  and $\underline{\lambda}_{t}\in\mathbb{R}_{+}$.
\begin{subequations}
\label{ForcingConstraints}
\begin{align}
s_{i} &\geq \underline{s}_{i},\quad i\in\mathcal{S},\;\; (\underline{\lambda}_{i})\\
d_{j} &\geq \underline{d}_{j},\quad j\in\mathcal{D},\;\; (\underline{\lambda}_{j})\\
f_{l} &\geq \underline{f}_{l},\quad l\in\mathcal{L},\;\; (\underline{\lambda}_{l})\\
c_{t} &\geq \underline{c}_{t},\quad t\in\mathcal{T},\;\; (\underline{\lambda}_{t})
\end{align}
\end{subequations}

Including any of the constraints of the type in \eqref{ForcingConstraints} in the clearing model will destroy some key economic properties but will not affect others. We first examine the nature of these constraints and their impacts on the clearing outcomes through the Lagrangian dual formulation, which will reveal that the profit-seeking nature of the coordinated market is unaffected by the forced participation constraints. Beginning from the Lagrangian function in \eqref{FinalLagrangian} we introduce the forcing constraints in \eqref{LagrangianForced}.

\begin{equation}
\label{LagrangianForced}
\begin{aligned}
\mathcal{L}(s,d,f,c,g,\pi)=-\left(\sum_{j \in \mathcal{D}}{\phi_{j}^{d}} + \sum_{i \in \mathcal{S}}{\phi_{i}^{s}} + \sum_{l\in\mathcal{L}}{\phi^{f}_{l}} + \sum_{t\in\mathcal{T}}{\phi^{c}_{t}}\right)\\
+\sum_{j \in \mathcal{D}}(\underline{d}_{j}-d_{j})\underline{\lambda}_{j} + \sum_{i \in \mathcal{S}}(\underline{s}_{i}-s_{i})\underline{\lambda}_{i} + \sum_{l \in \mathcal{L}}(\underline{f}_{l}-f_{l})\underline{\lambda}_{l} + \sum_{t \in \mathcal{T}}(\underline{c}_{t}-c_{t})\underline{\lambda}_{t}
\end{aligned}
\end{equation}

The terms associated with the forced participation constraints are such that either the constraint is active and has zero slack (i.e., $(\underline{s}_{i}-s_{i})=0$, for suppliers) or the constraint is inactive and the associated dual ($\underline{\lambda}_{i}$, for suppliers) is zero. In either case, the associated terms in \eqref{LagrangianForced} will evaluate to zero, so the profit-seeking property is preserved.

To illustrate the economic properties of the market clearing problem with forced participation, we once again turn to the dual formulation in \eqref{LPDualForced}. Here, the dual variables of forced participation constraints appear in the both the objective function and the constraints.
\begin{subequations}\label{LPDualForced}
\begin{equation}
\label{dualObjCondensedForced}
\begin{aligned}
\min_{\pi,\lambda}\; 
\sum_{i \in \mathcal{S}}{\overline{s}_{i}\overline{\lambda}_{i}} +
\sum_{j \in \mathcal{D}}{\overline{d}_{j}\overline{\lambda}_{j}} +
\sum_{l \in \mathcal{L}}{\overline{f}_{l}\overline{\lambda}_{l}} +
\sum_{t\in\mathcal{T}}{\overline{\xi}_{t}\overline{\lambda}_{t}}\\
-\sum_{i \in \mathcal{S}}{\underline{s}_{i}\underline{\lambda}_{i}} -
\sum_{j \in \mathcal{D}}{\underline{d}_{j}\underline{\lambda}_{j}} -
\sum_{l\in\mathcal{L}}{\underline{f}_{l}\underline{\lambda}_{l}} -
\sum_{t\in\mathcal{T}}{\underline{\xi}_{t}\underline{\lambda}_{t}}
\end{aligned}
\end{equation}
\begin{equation}
\label{dualCon_sCondensedForced}
\begin{aligned}
\pi^{bal}_{n(i),p(i)} - \overline{\lambda}_{i} + \underline{\lambda}_{i}\geq \alpha_{i}^{s},\; i\in\mathcal{S}
\end{aligned}
\end{equation}
\begin{equation}
\label{dualCon_dCondensedForced}
\begin{aligned}
\pi^{bal}_{n(j),p(j)} + \overline{\lambda}_{j}  - \underline{\lambda}_{j}\leq \alpha_{j}^{d},\; j\in\mathcal{D}
\end{aligned}
\end{equation}
\begin{equation}
\label{dualCon_fCondensedForced}
\begin{aligned}
\pi^{bal}_{n(l),p(l)} - \pi^{bal}_{n'(l),p(l)} - \overline{\lambda}_{l}  + \underline{\lambda}_{l}\geq \alpha^{f}_{l},\; (l)\in\mathcal{L}
\end{aligned}
\end{equation}
\begin{equation}
\label{dualCon_xCondensedForced}
\begin{aligned}
\sum_{p\in\mathcal{P}^{gen}_{t}}{\gamma_{t,p}\pi^{bal}_{n(t),p(t)}} - \sum_{p\in\mathcal{P}^{con}_{t}}{\gamma_{t,p}\pi^{bal}_{n(t),p(t)}} - \overline{\lambda}_{t}  + \underline{\lambda}_{t}\geq \alpha^{\xi}_{t},\; t\in\mathcal{T}
\end{aligned}
\end{equation}
\end{subequations}

We use the price identities to rewrite the dual constraints \eqref{dualCon_sCondensedForced} to \eqref{dualCon_xCondensedForced} (in terms of stakeholder prices $(\pi_{i},\pi_{j},\pi_{l},\pi_{t})$) and this results in \eqref{DualPricesForced}. Importantly, in all of these constraints, we observe that the signs corresponding to the dual variables $(\underline{\lambda}_{i},\underline{\lambda}_{j},\underline{\lambda}_{l},\underline{\lambda}_{t})$ are such that each relaxes the bound associated with the corresponding stakeholder bid $(\alpha^{s}_{i},\alpha^{d}_{j},\alpha^{f}_{l},\alpha^{\xi}_{t})$.
\begin{subequations}\label{DualPricesForced}
\begin{equation}
\label{dualCon_sIdentForced}
\begin{aligned}
\pi_{i} - \overline{\lambda}_{i} + \underline{\lambda}_{i}\geq \alpha_{i}^{s},\; i\in\mathcal{S}
\end{aligned}
\end{equation}
\begin{equation}
\label{dualCon_dIdentForced}
\begin{aligned}
\pi_{j} + \overline{\lambda}_{j} - \underline{\lambda}_{j}\leq \alpha_{j}^{d},\; j\in\mathcal{D}
\end{aligned}
\end{equation}
\begin{equation}
\label{dualCon_fIdentForced}
\begin{aligned}
\pi_{l} - \overline{\lambda}_{l} + \underline{\lambda}_{l}\geq \alpha^{f}_{l},\; l\in\mathcal{L}
\end{aligned}
\end{equation}
\begin{equation}
\label{dualCon_xIdentForced}
\begin{aligned}
\pi_{t} - \overline{\lambda}_{t} + \underline{\lambda}_{t}\geq \alpha^{\xi}_{t},\; t\in\mathcal{T}
\end{aligned}
\end{equation}
\end{subequations}

Strong duality implies that when the forced participation constraints are active, the corresponding duals $(\underline{\lambda}_{i},\underline{\lambda}_{j},\underline{\lambda}_{l},\underline{\lambda}_{t})$ will be non-negative. Furthermore, let us assume that $\overline{s}_{i}>\underline{s}_{i}$ so that $\overline{\lambda}_{i}=0$ and let the corresponding relationships apply such that $\overline{\lambda}_{j}$, $\overline{\lambda}_{l}$, and $\overline{\lambda}_{t}$ are also zero under active forced participation (i.e., when the forcing constraints are active and the upper bounding constraints are inactive). When the forced participation constraints are active, the duals $(\underline{\lambda}_{i},\underline{\lambda}_{j},\underline{\lambda}_{l},\underline{\lambda}_{t})$ allow stakeholder bid bounds to be violated according to $(\pi_{i}-\alpha^{s}_{i})\geq-\underline{\lambda}_{i}$,  $(\alpha^{d}_{j}-\pi_{j})\geq-\underline{\lambda}_{j}$,  $(\pi_{l}-\alpha^{f}_{l})\geq-\underline{\lambda}_{l}$, and $(\pi_{t}-\alpha^{\xi}_{t})\geq-\underline{\lambda}_{t}$, and from these results we observe the price bounding properties are lost. Importantly, this violation destroys the guarantee of having non-negative profits for all stakeholders (stakeholders can lose money in the market). 

 The previous results lead to interesting implications about ISO behavior under forced participation. The ISO does not need to recognize price bounds for such stakeholders; as such, it can raise or lower prices in violation of their bids to ensure that markets clear (even if they are allocated negative profits). The interpretation for a consumer compelled to purchase is that the ISO can raise the market price above the consumer bid according to $\pi_{j}\leq\alpha^{d}_{j}+\underline{\lambda}_{j}$. The consumer is compelled to accept a market price above its bid. Similarly, a supplier forced to participate in a market may be subject to prices below its service bid according to $\pi_{i}\geq\alpha^{s}_{i}-\underline{\lambda}_{i}$.  From an {\em efficiency} perspective, note that forcing participation implicitly forces the ISO to use services or products that are inefficient (e.g., forces use of a technology even if this is too costly).

We have observed that profit seeking is maintained in the market under forced participation, so Theorem~\ref{ThmProfit} applies to stakeholders unconstrained by minimum allocations. The profit maximizing property does not apply to stakeholders compelled to participate. Due to the limited application of profit seeking, Theorem~\ref{ThmComp} will not be satisfied either, as it is predicated on individual profit maximization. As such, the interpretation of the solution as a competitive equilibrium is lost under forced participation. 

Interestingly, revenue adequacy (Theorem~\ref{RevAqc}) still holds under forced participation since it is predicated on satisfaction of the market clearing constraints. To see this,  consider an optimal set of allocations $(s,\underline{s},d,\underline{d},f,\underline{f},c,\underline{c},g,\underline{g})$ for the market clearing problem subject to enforced allocations ${\underline{s},\underline{d},\underline{f},\underline{c},\underline{g}}$. The clearing constraints hold at optimality, implying:
	\begin{equation*}
	\begin{aligned}
	s_{n,p} + \underline{s}_{n,p} + \sum_{l\in\mathcal{L}^{in}_{n}}{f_{l}} + \sum_{l\in\mathcal{L}^{in}_{n}}{\underline{f}_{l}} +
	\sum_{t\in\mathcal{T}^{gen}_{n,p}}{g_{t}} + \sum_{t\in\mathcal{T}^{gen}_{n,p}}{\underline{g}_{t}}\\ - d_{n,p} - \underline{d}_{n,p} - \sum_{l\in\mathcal{L}^{out}_{n}}{f_{l}} - \sum_{l\in\mathcal{L}^{out}_{n}}{\underline{f}_{l}} - \sum_{t\in\mathcal{T}^{con}_{n,p}}{c_{t}} - \sum_{t\in\mathcal{T}^{con}_{n,p}}{\underline{c}_{t}}=0,
	\end{aligned}
	\end{equation*}
	Applying the market price identities, we express the dual multiplier $\pi^{bal}_{n,p}$ in terms of prices $\pi_{i}$, $\pi_{j}$, $\pi_{l}$, and $\pi_{t}$, and also rewrite technology inputs and outputs in terms of reference products $\bar{p}(t)$. Let us define notation for sets of forced allocations as $\check{\mathcal{S}}$ and unforced allocations as $\dot{\mathcal{S}}$ (using the same marks for each stakeholder class). Now we write:
	\begin{equation*}
	\begin{aligned}
	\sum_{i\in\dot{\mathcal{S}}}{\pi_{i}s_{i}}+\sum_{i\in\check{\mathcal{S}}}{\pi_{i}\underline{s}_{i}}+\sum_{l\in\dot{\mathcal{L}}}{\pi_{l}f_{l}}+\sum_{l\in\check{\mathcal{L}}}{\pi_{l}\underline{f}_{l}}+\sum_{t\in\dot{\mathcal{T}}}{\pi_{t}c_{t}}+\sum_{t\in\check{\mathcal{T}}}{\pi_{t}\underline{c}_{t}}-\sum_{j\in\dot{\mathcal{D}}}{\pi_{j}d_{j}}-\sum_{j\in\check{\mathcal{D}}}{\pi_{j}\underline{d}_{j}}=0.
	\end{aligned}
	\end{equation*}
This relationship demonstrates that revenue adequacy holds in the presence of forced allocations. It follows that revenue collected from consumers is sufficient to make payments to stakeholders under relaxed price bounding properties.

The results of Theorem~\ref{ThmNoCycles} will be lost as well under forced participation. Specifically, forcing participation allows for potential transport costs that are zero or negative and can create inefficient routes (cycles). This result intuitively makes sense since forcing (say demand satisfaction) will force transportation players to provide services (even if such services are inefficient). We can thus see that forcing participation of even one player can propagate through the market to activate it but this can create widespread inefficiencies. 

\section{Market-Activating Bids}\label{MinClearBids}

Motivated by the loss of economic properties resulting from forced market participation, we are interested in activating a market through economic incentives. Specifically, such incentives are provided by bidding costs. Our goal is to determine market-activating bid values, by which we mean the set of bids from consumers and suppliers capable of clearing market transactions with zero profit (i.e., these are break-even bids). Such bids define threshold market clearing prices and market existence criteria. Our interest in zero-profit markets is also driven by an inherent property of the market clearing framework: the distribution of profits among cleared stakeholders is often degenerate (there exist multiple ways to allocate the profit between stakeholders). A solution with zero profit avoids the profit allocation problem, even if it does not define unique allocations. 

Our results provide intuition for policy makers to strategically determine how to activate markets by either entering the market as stakeholders or by providing incentives that shift stakeholder bids (e.g., provide credits).  For instance, in a nutrient-trading market, a government agency that represents a lake can offer to allow nutrient emissions but subject to a tax that covers remediation costs associated with algal blooms. This program (if the emissions cost is sufficiently high) will create an incentive for startup technologies that offer services to recover nutrients. These incentives provide an economic setting where new technologies can be profitable because they offer a solution that is less expensive than the alternative. In this case, the objective of imposing a cost on nutrient emissions is to incentivize the businesses that produce the emissions to pay for nutrient recovery technologies. The incentive structure activates the market for nutrient recovery. Market-activating bids provide incentives that create a favorable setting to invest in new technologies. This is precisely how coordinated electricity markets operate in the US; here, satisfaction of demands is enforced by defining a bid known as the value-of-lost-load (VOLL). This bid cost captures the value of electricity to society and is set to values that are as high as 1,000 USD per MWh (the average cost in the US is around 10 USD per MWh). Effectively, VOLL operates as a market-activating bid that promotes provision of electricity. This is used, for instance, as a mechanism to enforce provision of electricity in remote areas. Effectively, market activating bids enforce participation but this is done via economic incentives (not by enforcing constraints). This approach preserves the economic properties of the market. 

We define a market-activating bid for a consumer $j\in\mathcal{D}^{+}$ as $\underline{\alpha}^{d}_{j}$ and for a supplier $i\in\mathcal{S}^{-}$ as $\underline{\alpha}^{s}_{i}$. Under this conceptualization, the other stakeholder bids $\alpha^{s}_{i},\; i\in\mathcal{S}^{+}$, $\alpha^{d}_{j},\; j\in\mathcal{D}^{-}$, $\alpha^{f}_{l},\; l\in\mathcal{L}$, and $\alpha^{c}_{t},\; t\in\mathcal{T}$ are interpreted as constants used to determine $\underline{\alpha}^{d}_{j}$ and $\underline{\alpha}^{s}_{i}$. The stakeholders $(\mathcal{D}^{+}, \mathcal{S}^{-})$ represent revenue sources in a market and $(\mathcal{S}^{+}, \mathcal{D}^{-},\mathcal{L},\mathcal{T})$ represent revenue sinks. Here, $\mathcal{S}=\mathcal{S}^-\cup\mathcal{S}^+$ and $\mathcal{D}=\mathcal{D}^-\cup\mathcal{D}^+$. The bids $(\underline{\alpha}^{d},\underline{\alpha}^{s})$ define remuneration rates corresponding to the allocations $(s,d,f,c,g)$ resulting in clear transactions $(\mathcal{S}^{*},\mathcal{D}^{*},\mathcal{L}^{*},\mathcal{T}^{*})$ with profits $\phi=(\phi^{s},\phi^{d},\phi^{f},\phi^{c})=0$. We first demonstrate that market-activating bids will result in zero profits through a revenue adequacy argument. Subsequently, we demonstrate how market-activating bids can be calculated by exploiting the topology of the supply chain graph  in the form of what we call a {\em stakeholder graph}.

\subsection{Revenue Adequacy Interpretation of Market-Activating Bids}

We use the definition of revenue adequacy to demonstrate that market-activating bids provide zero profit while remunerating stakeholders.
\begin{theorem}
	\label{BERevAqc}
	The market-activating bids $(\underline{\alpha}^{d},\underline{\alpha}^{s})$ with associated optimal allocations $(d,s,f,c,g)$  for  $(\mathcal{S}^{*},\mathcal{D}^{*},\mathcal{L}^{*},\mathcal{T}^{*})$ provide the economic funds required to achieve revenue adequacy with zero profit.
\end{theorem}
\begin{proof}
	By definition, the market-activating bids provide zero profit; consequently, we substitute the bids $(\underline{\alpha}^{d},\underline{\alpha}^{s})$ and $(\alpha^{s},\alpha^{d},\alpha^{f},\alpha^{c})$ to obtain:
	\begin{equation*}
	\begin{aligned}
	\phi^{d}_{j}(\pi_{j},\underline{\alpha}^{d}_{j},d_{j}) &= (\underline{\alpha}^{d}_{j} - \pi_{j})d_{j}=0,\quad j\in\mathcal{D}^{+*}\\
	\phi^{s}_{i}(\pi_{i},\underline{\alpha}^{s}_{i},s_{i}) &= (\pi_{i} - \underline{\alpha}^{s}_{i})s_{i}=0,\quad i\in\mathcal{S}^{-*}\\
	\phi^{s}_{i}(\pi_{i},\alpha^{s}_{i},s_{i}) &= (\pi_{i} - \alpha^{s}_{i})s_{i}=0,\quad i\in\mathcal{S}^{+*}\\
	\phi^{d}_{j}(\pi_{j},\alpha^{d}_{j},d_{j}) &= (\alpha^{d}_{j} - \pi_{j})d_{j}=0,\quad j\in\mathcal{D}^{-*}\\
	\phi^{f}_{l}(\pi_{l},\alpha^{f}_{l},f_{l}) &= (\pi_{l} - \alpha^{f}_{l})f_{l}=0,\quad l\in\mathcal{L}^{*}\\
	\phi^{c}_{t}(\pi_{t},\alpha^{c}_{t},c_{t}) &= (\pi_{t} - \alpha^{c}_{t})c_{t}=0,\quad t\in\mathcal{T}^{*}
	\end{aligned}
	\end{equation*}
	Since the allocations $(d,s,f,c,g)$ are nonzero in cleared transactions, the profit relationships imply the market-activating pricing properties:
	\begin{equation*}
	\begin{aligned}
	\pi_{j} = \underline{\alpha}^{d}_{j},\quad j\in\mathcal{D}^{+*}\\
	\pi_{i} = \underline{\alpha}^{s}_{i},\quad i\in\mathcal{S}^{-*}\\
	\pi_{i} = \alpha^{s}_{i},\quad i\in\mathcal{S}^{+*}\\
	\pi_{j} = \alpha^{d}_{j},\quad j\in\mathcal{D}^{-*}\\
	\pi_{l} = \alpha^{f}_{l},\quad l\in\mathcal{L}^{*}\\
	\pi_{t} = \alpha^{c}_{t},\quad t\in\mathcal{T}^{*}
	\end{aligned}
	\end{equation*}
	We thus have that market prices for clear transactions are equal to the corresponding bids, including the market-activating bids. Using the definition of revenue adequacy in Theorem~\ref{RevAqc} and substituting the market-activating pricing results we obtain:
	\begin{equation*}
	\begin{aligned}
	\sum_{j\in\mathcal{D}^{+}}{\underline{\alpha}^{d}_{j}d_{j}}-\sum_{i\in\mathcal{S}^{-}}{\underline{\alpha}^{s}_{i}s_{i}}
	=\sum_{i\in\mathcal{S}^{+}}{\alpha^{s}_{i}s_{i}}-\sum_{j\in\mathcal{D}^{-}}{\alpha^{d}_{j}d_{j}}+\sum_{l\in\mathcal{L}}{\alpha^{f}_{l}f_{l}}+\sum_{t\in\mathcal{T}}{\alpha^{c}_{t}c_{t}}
	\end{aligned}
	\end{equation*}
	demonstrating that market-activating bids provide funds to ensure revenue adequacy with zero stakeholder profit.
\end{proof}

The intuition behind this result is that stakeholders $(\mathcal{S}^{+},\mathcal{D}^{-},\mathcal{L},\mathcal{T})$ in a coordinated market require a given level of remuneration to participate, and this remuneration may be obtained from the stakeholders $(\mathcal{D}^{+},\mathcal{S}^{-})$ (these can be seen as players that drive the market). We now provide a procedure to determine market-activating bids.

\subsection{Computing Market-Activating Bids}

Market-activating bid calculations are based on the topology of the SC graph. To properly define these calculations, we turn to graph theoretical properties of SC networks. To facilitate our analysis, we propose the concept of the {\em stakeholder graph}; this graph is a product-based representation of the SC (typically analyzed using a node-based representation) that reveals how products connect stakeholders. Specifically, the stakeholder graph captures how products and associated economic values flow from suppliers to consumers through transformation and transportation steps. The stakeholder graph will also reveal natural ordering and boundaries of the SC (defined by suppliers and consumers) and its internal steps (transportation and transformation). 

\begin{definition}{\bf Stakeholder Graph.}	 The stakeholder graph associated with a SC is denoted as $G(V,E)$; here, the stakeholders (suppliers, consumers, transportation providers, and technology providers) are defined as nodes $V=\{\mathcal{S},\mathcal{D},\mathcal{L},\mathcal{T}\}$  and the arcs $E$ connect the players as follows: 
\begin{itemize}
\item A supplier $i$ and consumer $j$ are connected if $p(i)=p(j)$ and $n(i)=n(j)$.
\item A supplier $i$ and transport provider $l$ are connected if $p(i)=p(l)$ and $n_{s}(l)=n(i)$ (the case $n_{r}(l)=n(i)$ does not make sense because a supplier cannot receive product). 
\item A consumer $j$ and transport provider $l$ are connected if $p(j)=p(l)$ and $n_{r}(l)=n(j)$ (the case $n_{s}(l)=n(j)$ does not make sense because a consumer cannot send product).
\item A supplier $i$ and technology provider $t$ are connected if $p(i)\in \mathcal{P}^{con}_t$  (the case $p(i)\in \mathcal{P}^{gen}_t$ does not make sense because the supplier cannot receive product).  
\item A consumer $i$ and technology provider $t$ are connected if $p(j)\in \mathcal{P}^{gen}_t$ (the case $p(j)\in \mathcal{P}^{con}_t$ does not make sense because the consumer cannot send product).  
\item A transportation provider $l$ and a technology provider $t$ are connected if either: i) $p(l)\in \mathcal{P}^{con}_t$ and $n_{r}(l)=n(t)$ or ii) $p(l)\in \mathcal{P}^{gen}_t$ and $n_{s}(l)=n(t)$.
\item Suppliers $i$ and $i'$ cannot be connected (implies one receives product).
\item Consumers $j$ and $j'$ cannot be connected (implies one provides product).
\item Transport providers $l$ and $l'$ are connected if either: i) $n_{s}(l)=n_{r}(l')$ or ii) $n_{r}(l)=n_{s}(l')$.
\item Technology providers $t$ and $t'$ are connected if $n(t)=n(t')$ and either:  i) there exists $p\in \mathcal{P}_t^{gen}$ such that $p\in\mathcal{P}_{t'}^{con}$ or ii) there exists $p\in \mathcal{P}_t^{con}$ such that $p\in\mathcal{P}_{t'}^{gen}$.
\end{itemize}
\end{definition}

Note that the above definitions imply that the boundary of the stakeholder graph $\mathcal{B}(G)\subseteq V$ is such that $i\in \mathcal{B}(G)$ and $j\in \mathcal{B}(G)$ but $l\notin \mathcal{B}(G)$ and $t\notin \mathcal{B}(G)$ . In other words, suppliers and consumers define the boundaries and technology and transport providers are internal nodes. The stakeholder graph reveals important rules of the market architecture that are not immediately obvious from the traditional node-based representation. We now proceed to show that the abstraction allows us to define market-activating bids by following paths that connect the boundaries of the graph. The stakeholder graph can be seen as an alternative representation of the so-called p-graph used in process engineering \cite{Heckl2010}. In the stakeholder graph, however, we do not define product nodes (as products define stakeholder graph connectivity) and we capture transportation.

We illustrate the utility of the stakeholder graph representation of the SC in Figure~\ref{fig_comparison}. The product-based representation (figure center) reveals a directed acyclic graph (DAG) structure of the SC which is not obvious from the node-based representation (left). In addition, the product-based representation reveals how products flow between boundaries (from suppliers to consumers) through internal steps (transportation and transformation). We emphasize this representation using a set of eleven products in Figure~\ref{fig_comparison} (right) demonstrating that products define the connectivity between stakeholders. Paths through the stakeholder graph branch at some technologies (because they can produce multiple products or there are multiple consumers of a product) or collapse branches (because they consume multiple products or there are multiple suppliers of a product).

\begin{figure}[!htb]
	\center{\includegraphics[width=0.95\textwidth]{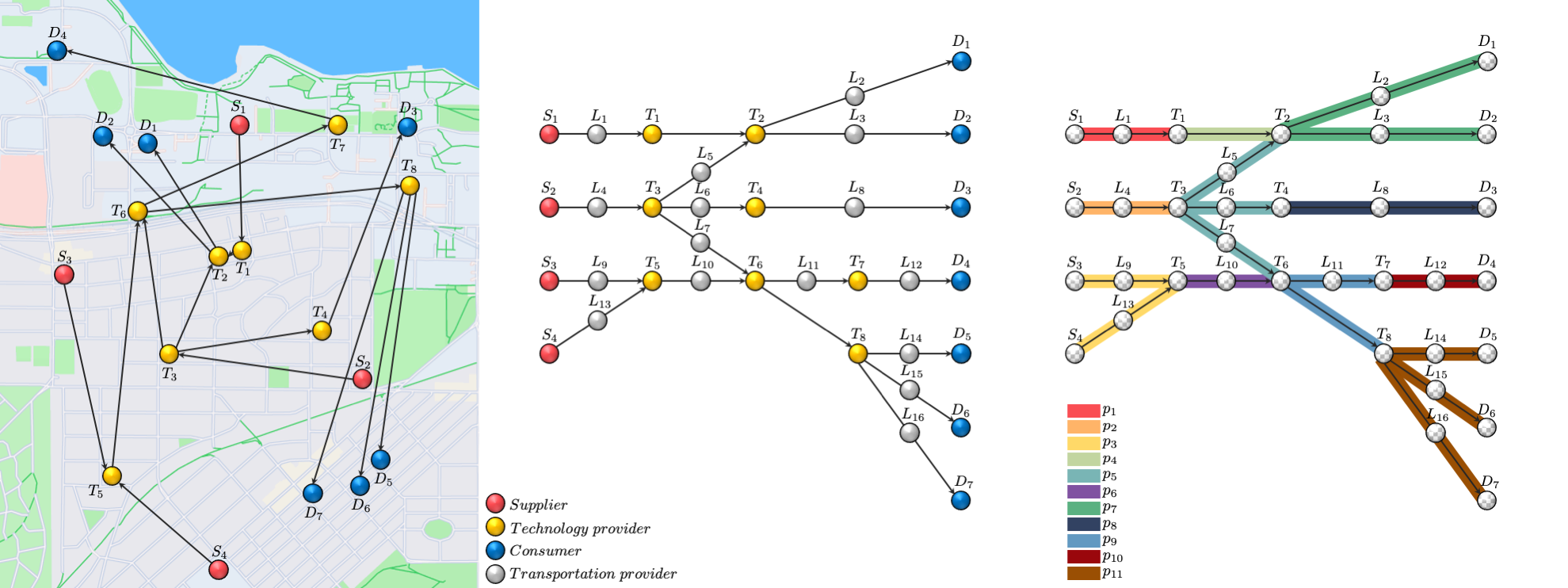}}
	\caption{Node-based (left) and product-based (center, right) representation of a supply chain. The product-based abstraction (stakeholder graph) elucidates paths in the SC that connect suppliers to consumers via transport and transformation steps. The center figure illustrates connections between stakeholders. The right figure emphasizes that products define connections in the stakeholder graph.}
	\label{fig_comparison}
\end{figure}

We proceed to calculate market-activating bids in terms of paths between nodes in the stakeholder graph representation. A path is a series of distinct connected nodes in a graph. Each node in the stakeholder graph has an associated stakeholder type (a supplier, a consumer, a transport, or a technology). All stakeholder graph nodes retain the associated stakeholder bid $(\alpha^{s}_{i},\alpha^{d}_{j},\alpha^{f}_{l},\alpha^{c}_{t})$ as an attribute. Technology nodes in the stakeholder graph also retain the technology yield coefficients $\gamma_{t,p}$ as attributes. Arcs in this representation are unweighted, capturing only topological (connectivity) information. As a convenient shorthand, we will denote the set of arcs in a stakeholder graph $\{\mathcal{N},\mathcal{P}\}$ to condense the list of attributes. All flows in the stakeholder graph are directed from suppliers towards technology providers or consumers. The graph may be partitioned into layers of suppliers, technology providers, transportation providers, and consumers. That such a partitioning can be achieved will allow us to express market-activating bids with respect to paths through the stakeholder graph. The paths of interest connect a stakeholder $u\in V$ (from whom the ISO will collect revenue) to a stakeholder $v\in V$ (to whom the revenue will be allocated). Stakeholders that activate the market (by providing revenue) are defined as $u\in\{\mathcal{D}^{+},\mathcal{S}^{-}\}$ with $\mathcal{D}^+\subseteq \mathcal{D}$ and $\mathcal{S}^-\subseteq\mathcal{S}$. The stakeholders reached by the activation (to whom revenue is allocated) are defined as $v\in\{\mathcal{D}^{-},\mathcal{S}^{+},T,L\}$. For convenience, we define the set of market-activating stakeholders as $\mathcal{A}:=\{\mathcal{D}^{+},\mathcal{S}^{-}\}$ and the set of reached stakeholders as $\mathcal{R}:=\{\mathcal{D}^{-},\mathcal{S}^{+},T,L\}$. We denote a path from a stakeholder $u$ to $v$ as $W$ and use $\mathcal{W}(u,v)$ to denote the set of all paths that connect $u$ and $v$.  

Central to our analysis is the observation that the stakeholder graph will tend to have a DAG structure (it has no product cycles).  The DAG structure will be key in analyzing how allocations and economic value flow through the SC. We will seek to establish conditions under which the presence of product cycles in a market clearing allocation implies that the allocation is infeasible (non-physical). In doing so, we note that product cycles can only be induced internally in the network by connecting technologies via cycles. Consequently, we will seek to follow a products along such technology cycles. 

	\begin{figure}[!htb]
		\center{\includegraphics[width=0.5\textwidth]{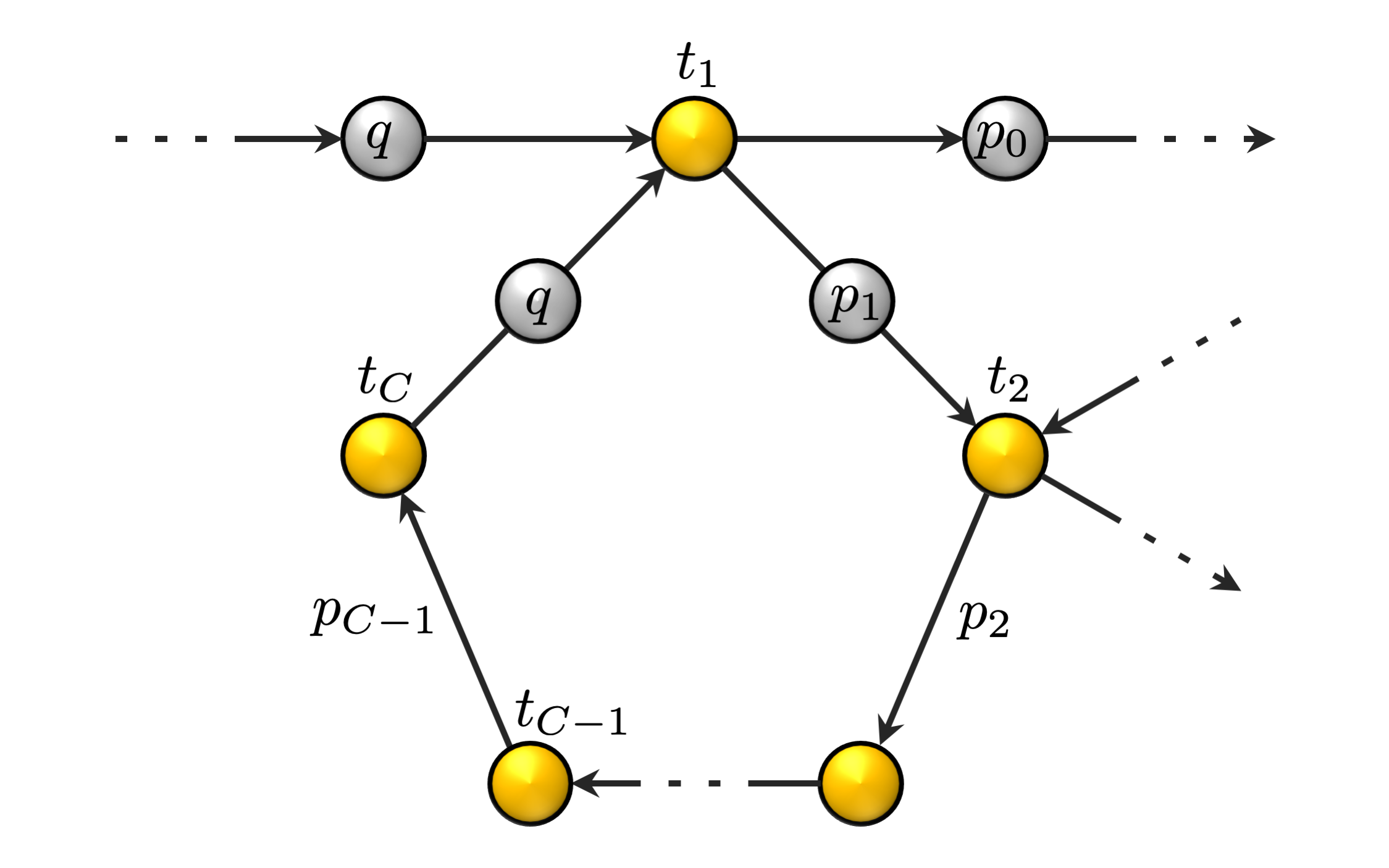}}
		\caption{Example supply chain structure with a technology cycle used in the proof of Theorem \ref{consistency}. The cycle eventually converts a output product of a technology back into one of its input products. The cycle may contain transportation cycles or other input and output products along the cycle (this does not change the conclusions of the analysis). Product transport is indicated by a product label inside a transport node (in gray) products exchanged directly between technology providers are indicated by a product label adjacent to the arrow connecting them.}
		\label{fig_CycleProof}
	\end{figure}

\begin{theorem}\label{consistency}
	The market clearing procedure can deliver non-physical allocations (infeasible to the market clearing problem) in the presence of technology cycles.
\end{theorem}
\begin{proof}
The proof is by contradiction. We establish an arbitrary technology cycle within an arbitrary SC. Without loss of generality, assume the cycle connects to the SC through a technology $t_{1}$ which has as its outputs at least two products, $p_{0}$ and $p_{1}$, and at least one input $q$. Define the technology cycle consisting of $C$ technologies such that $t_{2}$ has as an input product $p_{1}$ and produces product $p_{2}$, technology $t_{3}$ consumes $p_{2}$ and produces $p_{3}$, and so on, with $t_{C}$ consuming $p_{C-1}$ and producing product $q$ which can now be consumed by technology $t_{1}$. Assume that the ISO delivers a physical allocation in the SC containing such a cycle. Associated with each output from each technology is a yield coefficient $\gamma_{t,p}$ (we assume that input coefficients are all one without loss of generality) and with these assumptions we construct the constraints for the technology cycle. We assume that all technologies in the cycle are located at a single node $n$, and that node $n'$  is another SC node comprising a source of $q$ and a sink for $p_{0}$. 

We make no other assumptions about the nature of the SC, and observe that technology nodes in the cycle may be located at different geographical nodes (i.e., a transportation provider may be required within the cycle to move product between technologies) without changing the nature of our results. Moreover, elements of the cycle may have inputs or outputs from elsewhere in the SC. We depict the supply chain structure in Figure~\ref{fig_CycleProof}. Under this construct, we have the following balance and conversion constraints for the cycle.
	\begin{equation*}
	\begin{aligned}
	f_{n',n,q} + g_{t_{C},q} = c_{t_{1},q}\\ 
	g_{t_{1},p_{0}} = f_{n,n',p_{0}}\\ 
	g_{t_{1},p_{1}} = c_{t_{2},p_{1}}\\ 
	g_{t_{2},p_{2}} = c_{t_{3},p_{2}}\\ 
	\vdots\quad\\
	g_{t_{C-1},p_{C-1}} = c_{t_{C},p_{C-1}}\\ 
	\gamma_{t_{1},p_{0}}c_{t_{1},q} = g_{t_{1},p_{0}}\\ 
	\gamma_{t_{1},p_{1}}c_{t_{1},q} = g_{t_{1},p_{1}}\\ 
	\gamma_{t_{2},p_{2}}c_{t_{2},p_{1}} = g_{t_{2},p_{2}}\\ 
	\gamma_{t_{3},p_{3}}c_{t_{3},p_{2}} = g_{t_{3},p_{3}}\\ 
	\vdots\quad\\
	\gamma_{t_{C},q}c_{t_{C},p_{C-1}} = g_{t_{C},q}
	\end{aligned}
	\end{equation*}
	We repeat substitution of these constraints to define $g_{t_{C},q}$ in terms of product flows around the cycle. The procedure is illustrated as follows:
	\begin{equation*}
	\begin{aligned}
	g_{t_{C},q}&=\gamma_{t_{C},q}c_{t_{C},p_{C-1}}\\
	&=\gamma_{t_{C},q}g_{t_{C-1},p_{C-1}}\\
	&\quad\vdots\\
	&=\gamma_{t_{C},q}...\gamma_{t_{C},p_{3}}c_{t_{3},p_{2}}\\
	&=\gamma_{t_{C},q}...\gamma_{t_{3},p_{3}}g_{t_{2},p_{2}}\\
	&=\gamma_{t_{C},q}...\gamma_{t_{3},p_{3}}\gamma_{t_{2},p_{2}}c_{t_{2},p_{1}}\\
	&=\gamma_{t_{C},q}...\gamma_{t_{3},p_{3}}\gamma_{t_{2},p_{2}}g_{t_{1},p_{1}}\\
	&=\gamma_{t_{C},q}...\gamma_{t_{3},p_{3}}\gamma_{t_{2},p_{2}}\gamma_{t_{1},p_{1}}c_{t_{1},q}\\
	&=\gamma_{t_{C},q}...\gamma_{t_{3},p_{3}}\gamma_{t_{2},p_{2}}\gamma_{t_{1},p_{1}}(f_{n',n,q} + g_{t_{C},q})
	\end{aligned}
	\end{equation*}
	We denote the cumulative product yield $\gamma_{C} := \gamma_{t_{C},q}...\gamma_{t_{3},p_{3}}\gamma_{t_{2},p_{2}}\gamma_{t_{1},p_{1}}$, resulting in:
	\begin{equation*}
	\begin{aligned}
	g_{t_{C},q}\frac{1-\gamma_{C}}{\gamma_{C}}=f_{n',n,q}
	\end{aligned}
	\end{equation*}
	We have the cycle production rate of product $q$ ($g_{t_{C},q}$) in terms of the incoming rate of product $q$ to node $n$, $f_{n',n,q}$. The value of $\gamma_{C}$ defines the relationship between the two, and results in one of three possible outcomes:
	\begin{itemize}
		\item\label{item1} $\gamma_{C}>1$: implies negative flow, thus the only feasible solution for the market clearing problem is $f_{n',n,q}=g_{t_{C},q}=0$,
		\item\label{item2} $\gamma_{C}=1$: no feasible solution exists for market clearing for positive $f_{n',n,q}$; conversely, any positive $g_{t_{C},q}$ with $f_{n',n,q}=0$ is a feasible but non-physical solution,
		\item $\gamma_{C}<1$: there is a feasible solution for market clearing satisfying the relationship with positive $g_{t_{C},q}$ and $f_{n',n,q}$.
	\end{itemize}

	The existence of possible non-physical solutions for $\gamma_{C} \geq 1$ implies a contradiction. 
\end{proof}

The result indicates that existence of technology cycles in the market clearing allocations are determined by the cumulative yield in a technology cycle. A cumulative cycle yields that have a value larger than one leads to non-physical solutions, in particular, when cycle yield is equal to one, we have demonstrated that the ISO can allocate an arbitrary positive amount to the cycle while simultaneously allocating zero to its input. This result is non-physical because flows in the market clearing model must be traceable back to a supplier; other solutions imply a violation of the physical continuity of the model. Similarly, a cycle yield greater than one implies a negative flow, which is non-physical, but since a negative allocation is infeasible, the only feasible allocation is zero. The only way that the market clearing procedure delivers an allocation with cycles is if the cumulative yield is strictly less than one, which would imply that there is loss of product along the cycle. This indicates that the cycle provides an inefficient route.  We highlight that the previous theorem does not prove that this case cannot exist but it provides intuition as to why it might not occur. In practice, we have observed empirically that clearing allocations do not contain technology cycles. Formalizing this result is challenging due to interconnections of prices and physical allocations and will be left as a topic of future work. In the discussion that follows we will assume that the SC has a DAG structure. 

The stakeholder graph provides a node ordering which directs flows through paths of stakeholders from suppliers to consumers. Where multiple suppliers or consumers are present in a single stakeholder graph $G$, we are be able to partition $G$ into subgraphs $G'\subseteq G$. These subgraphs are proper components obtained by partitioning of $G$ along supplier and consumer vertices.
\begin{definition}{\bf Stakeholder Graph Component:}
	A stakeholder graph component is a subgraph $G'\subseteq G$ created by simultaneous partitioning along supplier and consumer vertices. The supply and demand nodes are included in all subgraphs created this way.
\end{definition}
Based on the definitions of the stakeholder graph and of the stakeholder graph component, we can establish at least one such component always exists. Specifically, consider the topology of an arbitrary graph $G$, either all supply and demand nodes in $G$ have degree one (whereby $G$ comprises one component) or there is at least one supply or demand node in $G$ with degree greater than one (in which case $G$ may comprise more than one component).

\begin{figure}[!htb]
	\center{\includegraphics[width=0.75\textwidth]{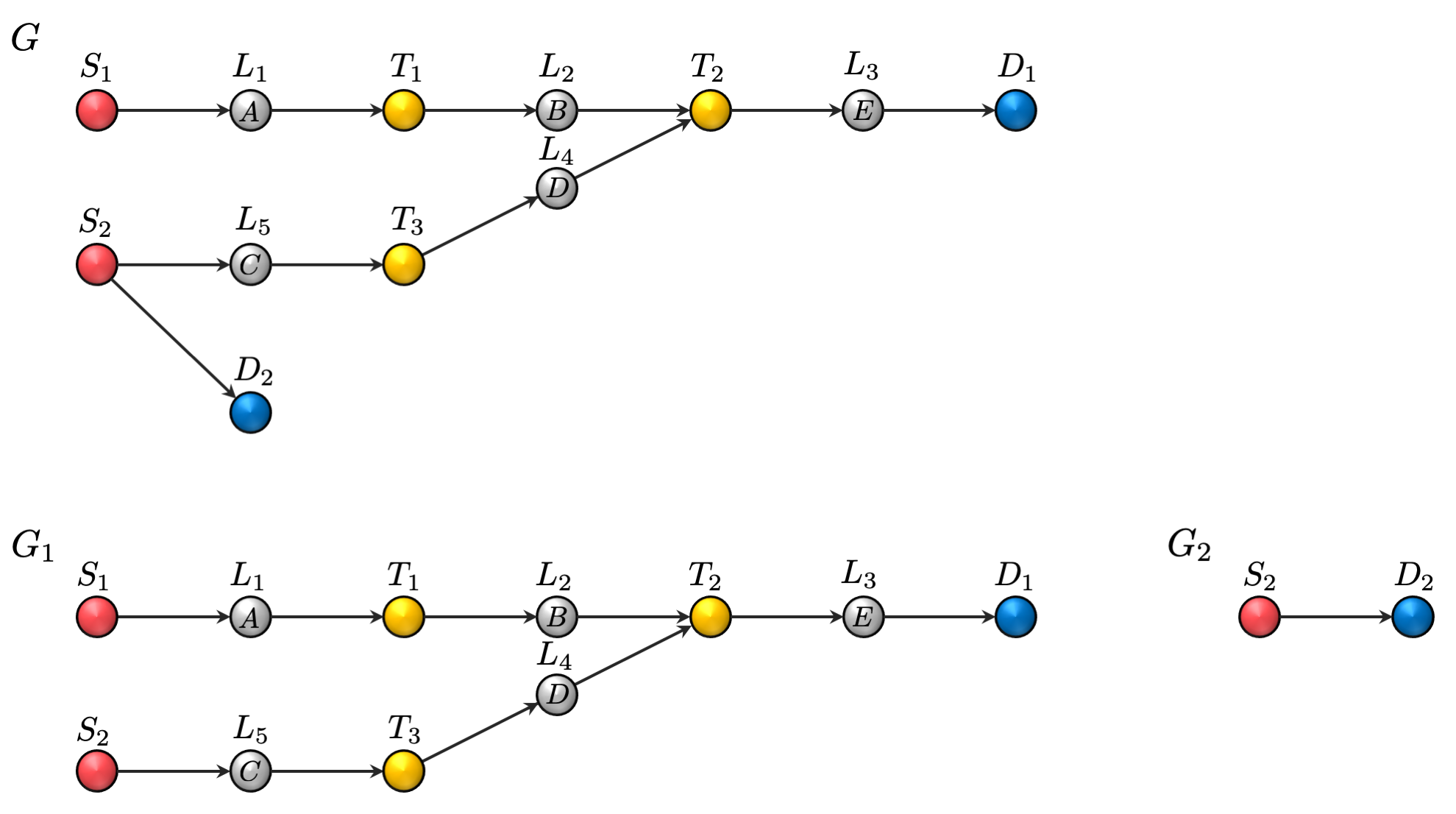}}
	\caption{Example stakeholder graph with two components. Partitioning $G$ along supplier node $S_2$ yields subgraphs $G_1$ and $G_2$ that can be analyzed independently.}
	\label{fig_Component}
\end{figure}

To illustrate the concept for the stakeholder graph component we introduce an example stakeholder graph $G$ in Figure~\ref{fig_Component} that contains two components ($G_1$ and $G_2$). The component subgraphs are obtained by partitioning of $G$ along the supplier node $S_2$, which has a degree of two. The components can be analyzed independently for purposes of determining bidding values.

The stakeholder graph provides a natural ordering of stakeholder nodes based on product precedence. Flows are always directed from supply nodes to consumers nodes and pass through transport and technology providers.  The relative positions of stakeholders $u\in\mathcal{A}$ and $v\in\mathcal{R}$ along the graph influence how remuneration is determined from $u$ to $v$. 

The remuneration rate for the activator is obtained by transforming the bid associated with stakeholder $v$ (defined for a given product $p(v)$) into the product basis corresponding to $u$ (given by $p(u)$) . Product transformations are associated with technology providers and are captured by the yield coefficients $\gamma_{t,p}$. We change the product basis by capturing all transformations along the path connecting $u$ and $v$ (i.e., for stakeholders $k\in W$). We denote the yield coefficients along the path $W$ as $\gamma^{gen}_{t(k),p(k)}$ for $p\in\mathcal{P}^{gen}_{t}$ and $\gamma^{con}_{t(k),p(k)}$ for $p\in\mathcal{P}^{con}_{t}$. For the purpose of the basis change, we can view a transport process as a technology with input and output yield coefficients of one (no transformation). If $u$ is upstream of $v$, then the correct product basis change is defined by the quotient of technology yield coefficients as:
\begin{equation}
\begin{aligned}
\frac{1}{\gamma_{v}^{con}}\prod_{k\in W}{\frac{\gamma^{gen}_{k}}{\gamma^{con}_{k}}},
\quad u\in \mathcal{A},v\in\mathcal{R},
\end{aligned}
\end{equation}
in which we note that having defined our stakeholder bids (in particular technology bids) relative to inputs, we exclude the output yield coefficient associated with $v$. In the case that $u$ is downstream of $v$ we have:
\begin{equation}
\begin{aligned}
\prod_{k\in W}{\frac{\gamma^{con}_{k}}{\gamma^{gen}_{k}}},
\quad u\in\mathcal{A},v\in\mathcal{R}.
\end{aligned}
\end{equation}
Without loss of generality, we will proceed by adopting the convention that $u$ is downstream of $v$.

With these definitions and results, we are now in a position to construct the {\em market-activating bid computation} procedure. A stakeholder $v\in\mathcal{R}$ in a market $G$ has an associated bid denoted $\alpha_{v}$ that defines the remuneration rate required for $v$ to participate in the market. The ISO can collect revenue to remunerate $v$ from stakeholders $u\in\mathcal{A}$ in the same graph component $G'\subseteq G$. Let us define $\beta_{u,v}\in[0,1]$ as the fraction remuneration for stakeholder $v$ provided by stakeholder $u$, subject to the criterion that $\sum_{u\in\mathcal{A}}{\beta_{u,v}}=1,\; v\in\mathcal{R}$. Stakeholders $u$ and $v$ are connected by a path $W\in\mathcal{W}(u,v)$. The product basis change along $W$ ultimately defines the bid required from stakeholder $u$ to remunerate stakeholder $v$ and is given by $\underline{\alpha}_{u,v}$, which we denote the partial bid in \eqref{PartialMinBid}.
\begin{equation}
\label{PartialMinBid}
\begin{aligned}
\underline{\alpha}_{u,v} = \beta_{u,v}\prod_{k\in W}{\frac{\gamma^{con}_{k}}{\gamma^{gen}_{k}}}\alpha_{v},\quad u\in\mathcal{A},v\in\mathcal{R}
\end{aligned}
\end{equation}
Combining all revenue commitments $\underline{\alpha}_{u,v}$ from a given stakeholder $u\in\mathcal{A}$ to stakeholders $v\in\mathcal{R}$ determines the market-activating bid for stakeholder $u$: $\underline{\alpha}_{u}$ \eqref{MinBid}. We will refer to $\underline{\alpha}_{u}$ as the stakeholder $u$ market-activating bid and to $\underline{\alpha}_{u,v}$ as its partial market-activating bids.
\begin{equation}
\label{MinBid}
\begin{aligned}
\underline{\alpha}_{u} = \sum_{v\in G'}\beta_{u,v}\prod_{k\in W}{\frac{\gamma^{con}_{k}}{\gamma^{gen}_{k}}}\alpha_{v},\quad u\in\mathcal{A}.
\end{aligned}
\end{equation}

\begin{figure}[!htb]
	\center{\includegraphics[width=0.75\textwidth]{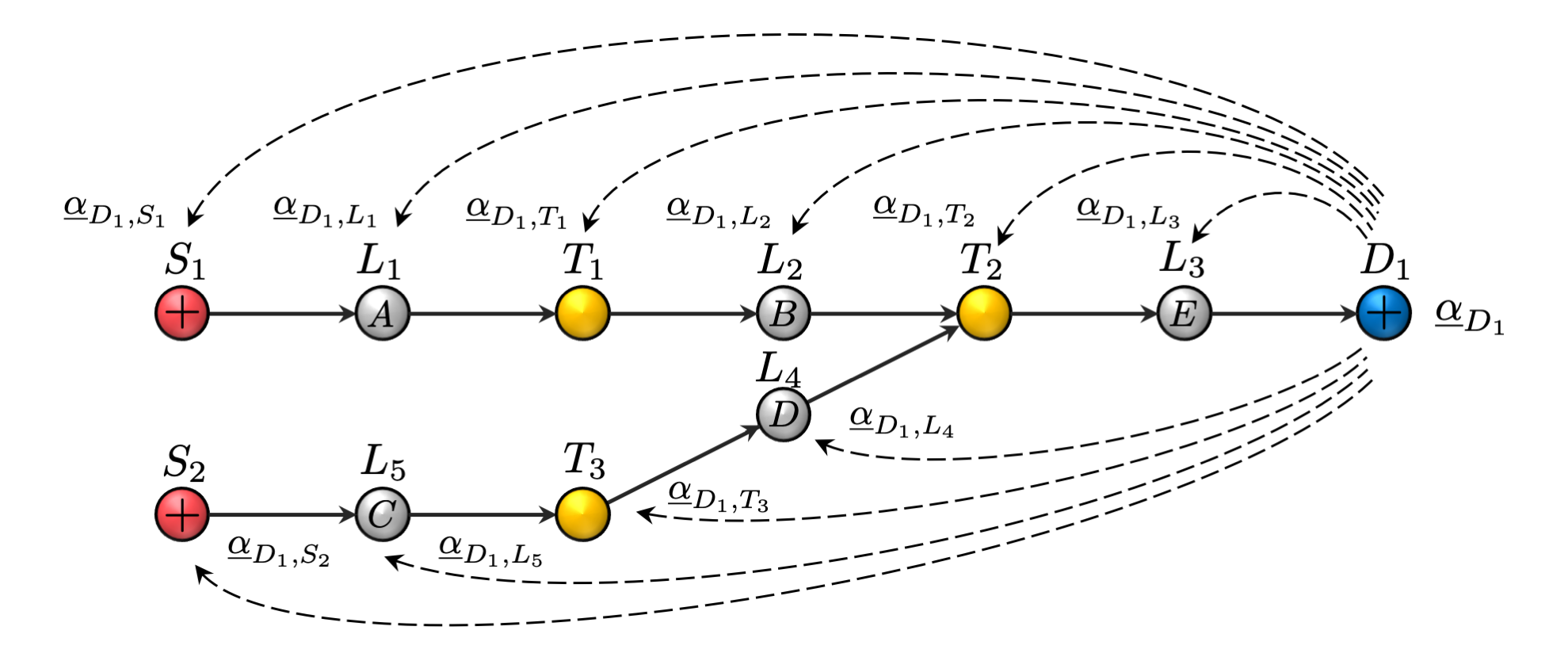}}
	\caption{Breakdown of consumer $D_1$ market-activating bid into partial bids corresponding to each remunerated stakeholder.}
	\label{fig_BidBreakdown}
\end{figure}

We illustrate the breakdown of a market-activating bid using the stakeholder graph component $G_1$ from Figure~\ref{fig_Component}. In Figure~\ref{fig_BidBreakdown}, we assume that $D_1\in\mathcal{D}^{+}$ is the market-activating player (i.e., a consumer with a positive is a source of market revenue) and $S_1,S_2\in\mathcal{S}^{+}$ are the reached players (i.e., suppliers with positive bids demand payment and are revenue sinks). In this setting, all the market revenue is obtained from $D_1$ and this has to be sufficient to activate all players along the suuply chain. The market-activating bid for $D_1$ is $\underline{\alpha}_{D_1}$. The breakdown of this bid into its partial bids corresponding to each remunerated stakeholder is indicated in Figure~\ref{fig_BidBreakdown}. Concretely, the partial bid from $D_1$ to $T_2$ is represented as:
\begin{equation*}
\begin{aligned}
\underline{\alpha}_{D_1,T_2} = {\frac{\gamma_{T_2,B}}{\gamma_{T_2,E}}}\alpha_{T_2},
\end{aligned}
\end{equation*}
Similarly, the remuneration rate to $L1$ is determined as:
\begin{equation*}
\begin{aligned}
\underline{\alpha}_{D_1,L1} = {\frac{\gamma_{T_2,B}}{\gamma_{T_2,E}}}{\frac{\gamma_{T_1,A}}{\gamma_{T_1,B}}}\alpha_{L1},
\end{aligned}
\end{equation*}
illustrating how the yield coefficients in technologies $T_1$ and $T_2$ influence the remuneration rate from $D_1$ to transport provider $L_1$. The yield factors accumulate along the path $W(D_1,L_1)$. The partial bid $\underline{\alpha}_{D_1,S_2}$ follows the path $W(D_1,S2)$, and is determined as
\begin{equation*}
\begin{aligned}
\underline{\alpha}_{D_1,S_2} = {\frac{\gamma_{T_2,D}}{\gamma_{T_2,E}}}{\frac{\gamma_{T_3,C}}{\gamma_{T_3,D}}}\alpha_{S_2},
\end{aligned}
\end{equation*}
where the yield coefficient $\gamma_{T_2,D}$ corresponds to the input product on the correct path. By construction, one product is defined as the reference input for each technology, and has a yield coefficient of one. Technologies with multiple input products will thus have one input with a coefficient of one and others with coefficients $\gamma_{t,p}\in\mathbb{R}_{+}$. 

\subsection{Properties of Market-Activating Bids}{\label{SecProperties}}

\subsubsection{Revenue Distribution}

The value $\beta_{u,v}$ represents the fractional revenue provided by stakeholder $u$ to $v$. We now establish that this fractioning of revenue over stakeholders $u$ provides the correct revenue to stakeholder $v$.
\begin{theorem}
	\label{ThmSharing}
	Consider a stakeholder graph $G$ with activators $\mathcal{A}$ and targets $\mathcal{R}$. Each stakeholder $v\in\mathcal{R}$ is part of one or more paths $W$ connecting to stakeholders $u\in\mathcal{A}$, and there exist $\beta_{u,v}\in[0,1]$ satisfying $\sum_{u\in\mathcal{A}}{\beta_{u,v}}=1$ such that the bids $\underline{\alpha}_{u,v}$ achieve revenue adequacy for stakeholder $v$.
\end{theorem}
\begin{proof}
The result follows  from Theorem \ref{BERevAqc}. The total revenue provided to the ISO by the market-activating stakeholders $u\in\mathcal{A}$ must be sufficient to remunerate all cleared stakeholders. Specifically, all revenue allocated to stakeholders $v\in\mathcal{R}$ is collected from stakeholders $u\in\mathcal{A}$. Let us select an arbitrary stakeholder $v\in \mathcal{A}$; we define partial bids $\alpha_{u,v}$ and optimal allocations $\mu_{u}$ for stakeholders $u$ and a service bid $\alpha_{v}$ and optimal allocation $\theta_{v}$ for stakeholder $v$. Revenue adequacy for $v$ is defined:
	\begin{equation*}
	\begin{aligned}
	\sum_{u\in\mathcal{A}}{\alpha_{u,v}\mu_{u}}=\alpha_{v}\theta_{v}.
	\end{aligned}
	\end{equation*}
	Revenue adequacy for stakeholder $v$ is linearly separable over the stakeholders  $u\in\mathcal{A}$ and the division of remuneration between the $u$ is degenerate (multiple divisions exist). We pick $\beta_{u,v}\in[0,1], u\in\mathcal{A}$ satisfying $\sum_{u\in\mathcal{A}}{\beta_{u,v}}=1$ allowing us to rewrite revenue adequacy as:
	\begin{equation*}
	\begin{aligned}
	\alpha_{u,v}\mu_{u}=\beta_{u,v}\alpha_{v}\theta_{v},\quad u\in\mathcal{A}
	\end{aligned}
	\end{equation*}
	Thus revenue adequacy is achieved and $v$ participates in the market for some distribution of remuneration defined by $\beta_{u,v}$.
\end{proof}

We observe that the ratio of $\theta_{v}$:$\mu_{u}$ is equal to the cumulative yield between stakeholders $u$ and $v$. Rearranging results in
	\begin{equation*}
	\begin{aligned}
	\alpha_{u,v}=\beta_{u,v}\frac{\theta_{v}}{\mu_{u}}\alpha_{v},\quad u\in\mathcal{A}.
	\end{aligned}
	\end{equation*}
	
Substitution of the equivalent product of input and output yield ratios for $\frac{\theta_{v}}{\mu_{u}}=\prod_{k\in W}{\frac{\gamma^{con}_{k}}{\gamma^{gen}_{k}}}$ recovers the calculation for the partial market-activating bid calculation $\underline{\alpha}_{u,v}$ in equation~\ref{PartialMinBid}.

We illustrate the distribution of stakeholder remuneration by modifying our example from Figure~\ref{fig_BidBreakdown}. In Figure~\ref{fig_BidDistribution} we now assume $D_1\in\mathcal{D}^{+}$ and $S_1,S_2\in\mathcal{S^{-}}$ where the consumer $D_1$ and suppliers $S_1$ and $S_2$ are all sources of market revenue (market-activating players). Figure~\ref{fig_BidDistribution} depicts the distribution of remuneration between them for technology $T_2$ and transport provider $L_1$. Market-activating remuneration to $T_2$ is defined by:
\begin{figure}[!htb]
	\center{\includegraphics[width=0.75\textwidth]{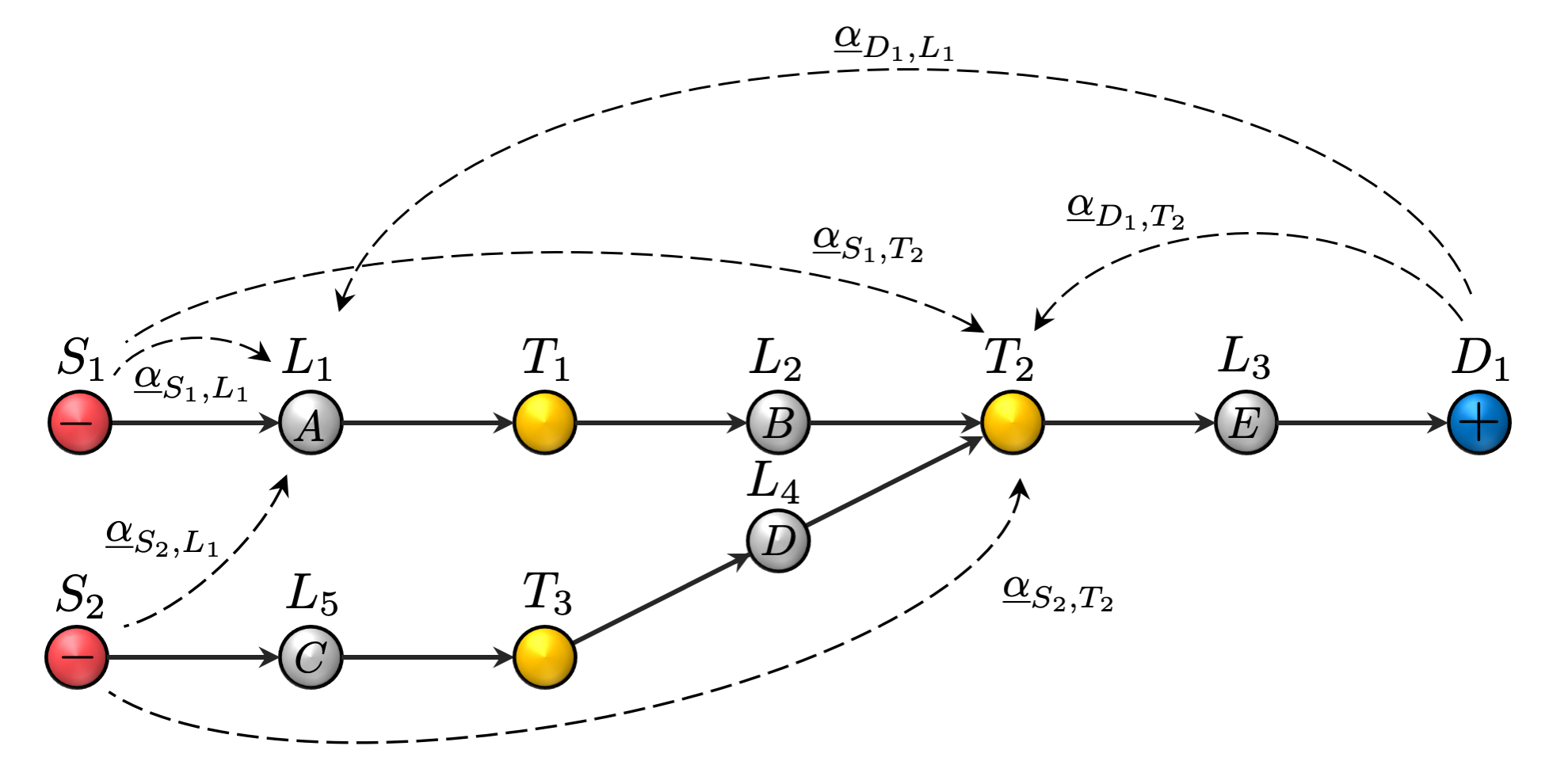}}
	\caption{Breakdown of consumer $D_1$ market-activating bid into partial bids corresponding to each remunerated stakeholder.}
	\label{fig_BidDistribution}
\end{figure}
\begin{equation*}
	\begin{aligned}
	\underline{\alpha}_{D_1,T_2} = \beta_{D_1,T_2}{\frac{\gamma_{T_2,B}}{\gamma_{T_2,E}}}\alpha_{T_2}\\
	\underline{\alpha}_{S_1,T_2} = \beta_{S_1,T_2}{\frac{\gamma_{T_1,B}}{\gamma_{T_1,A}}}\frac{1}{\gamma_{T_2,B}}\alpha_{T_2}\\
	\underline{\alpha}_{S_2,T_2} = \beta_{S2,T_2}{\frac{\gamma_{T_3,D}}{\gamma_{T_3,C}}}\frac{1}{\gamma_{T_2,D}}\alpha_{T_2}\\
	\beta_{D_1,T_2} + \beta_{S_1,T_2} + \beta_{S2,T_2} = 1,
	\end{aligned}
\end{equation*}
The remuneration distribution for stakeholder $L_1$ is more complex and involves computing a remuneration rate based on a path that travels both up and downstream in the stakeholder graph. The bid distribution for $L_1$ is defined by the set of equations
\begin{equation*}
\begin{aligned}
\underline{\alpha}_{D_1,T_2} &= \beta_{D_1,L_1}\frac{\gamma_{T_2,B}}{\gamma_{T_2,E}}\frac{\gamma_{T_1,A}}{\gamma_{T_1,B}}\alpha_{L1}\\
\underline{\alpha}_{S_1,T_2} &= \beta_{S_1,L_1}\alpha_{L_1}\\
\underline{\alpha}_{S_2,T_2} &= \beta_{S_2,L_1}{\frac{\gamma_{T_3,D}}{\gamma_{T_3,C}}}\frac{\gamma_{T_2,B}}{\gamma_{T_2,D}}\frac{\gamma_{T_1,A}}{\gamma_{T_1,B}}\alpha_{L_1}\\
&\beta_{D_1,L_1} + \beta_{S_1,L_1} + \beta_{S_2,L_1} = 1.
\end{aligned}
\end{equation*}
In each case, the $\beta$ coefficients (each in the range $[0,1]$) have a combined total of 1, ensuring that the total revenue received 

\subsubsection{Remuneration Properties over Multiple Pathways}

Equation \eqref{MinBid} implicitly assumes that the stakeholder graph features a single path $W$ between stakeholders $u$ and $v$. Multiproduct supply chains need not conform to this topology. Specifically, multi-input multi-output technologies can give rise to supply chains in which there are multiple paths between the stakeholders in the boundary. Consider the example in Figure~\ref{fig_MultiPath} in which two paths, $W_{B}(D_1,S_1)$ and $W_{D}(D_1,S_1)$ (indexed by the products produced by technology $T_1$ for convenience) are available between stakeholders $D_1$ and $S_1$. How the remuneration rate from $D_1$ to $S_1$ is determined is unclear from the analysis so far.

\begin{figure}[!htb]
	\center{\includegraphics[width=0.8\textwidth]{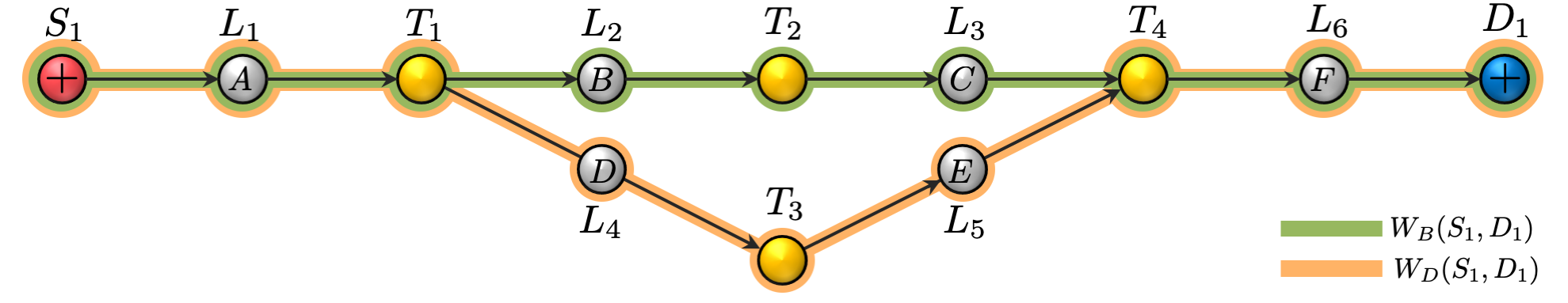}}
	\caption{Illustrative example for multiproduct SC with multiple paths between stakeholders. Two paths between stakeholders $S_1$ and $D_1$ exist, and are highlighted to illustrate the stakeholders included in each.}
	\label{fig_MultiPath}
\end{figure}

To deal with this more complex setting, we define the set of paths between one stakeholder $u$ and a second $v$ as $\mathcal{W}(u,v)$; a subset of all paths. Define the cumulative yield:
\begin{equation*}
\begin{aligned}
\hat{\gamma}_{W}:=\prod_{k\in W}{\frac{\gamma^{gen}_{k}}{\gamma^{con}_{k}}},
\end{aligned}
\end{equation*}
for all paths $W\in\mathcal{W}(u,v)$. The path $W$ which determines the remuneration rate from $u$ to $v$ is the one having the minimum cumulative yield $\hat{\gamma}_{W}$. This result follows from the physical requirements of multi-input multi-output technologies: the minimum $\hat{\gamma}_{W}$ corresponds to the limiting product $p\in\mathcal{P}^{con}_{t}$ among the inputs to a technology $t$ along the path $W$ (we define this product as $\hat{p}$). The supply chain needs to provide enough of this product to satisfy the input requirement at $t$, and other inputs will be provided in excess as a result.

\begin{theorem}
	Assume that $\mathcal{W}(u,v)$ is not a singleton (there exist multiple remuneration paths from stakeholder $u$ to stakeholder $v$). A revenue adequate remuneration rate for $v$ is determined according to the minimum path yield $\hat{\gamma}_{W}$.
\end{theorem}
\begin{proof}
The proof is by contradiction. Assume that we have stakeholders $u$ and $v$ connected by a set of paths $\mathcal{W}(u,v)$, and that the path yields $\hat{\gamma}_{W}$ are known for each path $W$. Also assume that there is a unique minimum path yield $\hat{\gamma}^{*}_{W}$ with corresponding path $W^{*}$. Now, let us select some other path $W$ with yield $\hat{\gamma}_{W}$, and calculate remuneration rates for stakeholder $v$ according to this path. Now claim that this remuneration rate is sufficient to incentivize $v$ to participate in the market  (i.e., leads to adequate remuneration). We proceed to show that this will result in a contradiction. Further, assume that $u$ is the sole source of revenue in the market and $\beta_{u,v}=1$ as a simplification. From the definition of path yields, we express the remuneration rate from $u$ to $v$ as:
	\begin{equation*}
	\begin{aligned}
	\underline{\alpha}_{u,v} = \frac{1}{\hat{\gamma}_{W}}\alpha_{v},
	\end{aligned}
	\end{equation*}
	Define the optimal allocations $\mu_{u}$ and $\theta_{v}$, for the stakeholders $u$ and $v$ and observe that revenue adequacy for stakeholder $v$ is:
	\begin{equation*}
	\begin{aligned}
	\underline{\alpha}_{u,v}\mu_{u} = \alpha_{v}\theta_{v},
	\end{aligned}
	\end{equation*}
	which we can express as
	\begin{equation*}
	\begin{aligned}
	\underline{\alpha}_{u,v} = \frac{\theta_{v}}{\mu_{u}}\alpha_{v}=\frac{1}{\hat{\gamma}^{*}_{W}}\alpha_{v},
	\end{aligned}
	\end{equation*}
	since this physical relationship is known. We have stated that $\hat{\gamma}^{*}_{W}<\hat{\gamma}_{W}$ by assumption, and thus
	\begin{equation*}
	\begin{aligned}
	\frac{1}{\hat{\gamma}^{*}_{W}}\alpha_{v}>\frac{1}{\hat{\gamma}_{W}}\alpha_{v}
	\end{aligned}
	\end{equation*}
	implying
	\begin{equation*}
	\begin{aligned}
	\underline{\alpha}_{u,v}>\underline{\alpha}_{u,v},
	\end{aligned}
	\end{equation*}
	which is a contradiction (the only remuneration rate that does not arise in a contradiction is that determined by $\hat{\gamma}^{*}_{W}$).

\end{proof}

The definition of path yield allows us to express market-activating bid calculations in the presence of multiple paths $W\in \mathcal{W}(u,v)$ as: 
\begin{equation}
\label{MultiPartialMinBid}
\begin{aligned}
\underline{\alpha}_{u,v} = \beta_{u,v}\max_{W\in \mathcal{W}(u,v)}{\prod_{k\in W}{\frac{\gamma^{con}_{k}}{\gamma^{gen}_{k}}}}\alpha_{v},\quad u\in\mathcal{A}, \, v\in \mathcal{R}.
\end{aligned}
\end{equation}
We observe that this generalizes the expression in \eqref{PartialMinBid} to multiple paths. For example, consider the SC in Figure~\ref{fig_MultiPath}. Assume $\hat{\gamma}_{B}=1$, and $\hat{\gamma}_{D}=0.5$, again using products to distinguish between the two paths. Regardless of the individual technology yield coefficients, $\hat{\gamma}_{D}$ tells us that we will require twice the supply of product $A$ from supplier $S_1$ on the path $W_{D}(D_1,S_1)$ than is required along the path $W_{B}(D_1,S_1)$, and that $S_1$ must be remunerated accordingly. This result also tells us that there will be an excess of product $B$. Since product $A$ is the precursor to both products $B$ and $D$, the supplier $S_1$, transport provider $L1$, and technology $T_1$ must all be remunerated according to the amount of product $A$ required on the path $W_{D}(D_1,S_1)$ determined by the path yield $\hat{\gamma}_{D}$.

\section{Case Study}\label{sec:problem}

\subsection{Problem Setting}
We illustrate the concepts discussed using a municipal solid waste (MSW) case study for a small-sized city of 100,000 citizens. These citizens are suppliers that pay for a waste collection service that transports their waste to a local landfill. Several recycling (processing) services exist within the city that convert specific waste products into value-added products. There is a sorting process available to separate waste streams and recycling services may purchase sorted waste as a feedstock from this service. The landfill charges a tipping fee to dispose of sorted or unsorted waste. Also included in the system is an electrical power producer that supplies electricity to city residents and to the recycling service providers. This setting defines a waste management economy in which recycling technologies provide a route to divert waste from landfills. We model this system as a multiproduct supply chain and use our coordinated market interpretation to investigate incentives that can activate recycling services and that foster landfill diversion. These results may be of interest to policymakers in similarly sized cities (e.g., the City of Madison, Wisconsin has a reported population on the order of 260,000 people). We provide a perspective on how to apply our insights to such settings. 

The US EPA estimates that the average citizen is responsible for the generation of 4.4 lbs of MSW per day~\cite{EPAMSW}. Globally, it is estimated that 1.5 billion tonnes of MSW are landfilled annually, producing widespread impacts (e.g., greenhouse gas emissions, leaching, land use). Encouraging MSW recycling can stimulate local economies and prevent environmental impacts associated with landfills\cite{EPASusReport}. Unfortunately, establishing waste markets that are financially sustainable is nontrivial due to conflicting interests of the stakeholders involved and due to the inherently low value of waste streams and recycled products. For instance, Vossberg and co-workers demonstrate that the market for recycled glass can be difficult to maintain because the recycling process (involving transportation and processing) is more expensive than the production of virgin glass~\cite{Vossberg2014}. Similar obstacles are encountered for other waste streams (e.g., food waste and plastic).

The EPA provides estimates for the composition of MSW consisting of nine categories by weight percent~\cite{EPAMSW}. For the purposes of this study the EPA's "food" and "lawn" categories are combined into "organic waste," and the categories "rubber, leather, and textiles," "wood," and "other" are grouped as "non-recyclable." The remaining categories are paper, glass, metals, and plastics, which are preserved. The resulting solid waste composition is presented in Table \ref{WasteData} and this includes the mass production rates corresponding to the selected population. Here we also include the possible recycled products that can be produced from each waste product and the index assigned to each product for the purposes of modeling notation.
\begin{table}.
	\centering
	\caption{Breakout of municipal waste composition for MSW case study.}
	\label{WasteData}
	\resizebox{\textwidth}{!}{%
	\begin{tabular}{cccccc}
		\hline
		\textbf{Waste category} & \textbf{Waste Label} & \textbf{Mass Fraction (\%)} & \textbf{Annual Production (tonne/yr)} & \textbf{Derived Products} & \textbf{Product Label}\\
		\hline
		Paper & P\textsubscript{01} & 27.0 & 19669 & Recycled Paper & P\textsubscript{1}\\
		\hline
		Glass & P\textsubscript{02} & 4.5 & 3278 & Recycled Glass & P\textsubscript{2}\\
		\hline
		Metals & P\textsubscript{03} & 9.1 & 6629 & Recycled Metals & P\textsubscript{3}\\
		\hline
		Plastics & P\textsubscript{04} & 12.8 & 9324 & Recycled Plastics & P\textsubscript{4}\\
		\hline
		Organic & P\textsubscript{05} & 28.1 & 20470 & Compost & P\textsubscript{5}\\
		\hline
		Non-recyclable & P\textsubscript{06} & 18.5 & 13477 & - & -\\
		\hline
		Total & P\textsubscript{0} & 100 & 72847 & - & -\\
		\hline
	\end{tabular}}
\end{table}

All process technology data are presented in Table \ref{ProcessData}. Recycling technologies for paper, glass, metal, and plastic wastes with associated yields and operating costs are obtained from Mohammadi \textit{et al.}~\cite{Mohammadi2019} and Santibanez \textit{et al.}~\cite{Santibanez2013}. Process details for organic waste composting are obtained from van Haaren \textit{et al.}~\cite{vanHaaren2010}. Yield factors for process electricity requirements are obtained from literature sources. Input and output process yields are combined into a conversion equation relating all inputs and outputs to a reference input such that the reference input has a yield coefficient of one.
\begin{table}.
	\centering
	\caption{Recycling process data for MSW case study. Yield values are in tonnes (except for electricity which is indicated in kWh).}
	\label{ProcessData}
	\resizebox{\textwidth}{!}{%
	\begin{tabular}{cccccccc}
		\hline
		\textbf{Process} &\textbf{Label}&\textbf{Inputs} & \textbf{Outputs} & \textbf{Input Yield} & \textbf{Output Yield} & \textbf{Operating Cost (USD)} & \textbf{Capacity (tonne)}\\
		\hline
		Separation~\cite{EPAMSW} & T\textsubscript{0} & P\textsubscript{0} & P\textsubscript{01}, P\textsubscript{02}, P\textsubscript{03}, P\textsubscript{04}, P\textsubscript{05}, P\textsubscript{06} & 1.00 & 0.270, 0.045, 0.091, 0.128, 0.281, 0.185 & 235.00~\cite{Santibanez2013} & 72847\\ 
		\hline
		Paper recycling & T\textsubscript{1} & P\textsubscript{01}, P\textsubscript{E} & P\textsubscript{1}, P\textsubscript{06} & 1.00, 300~\cite{Merrild2008} & 0.84,0.16 & 210.90 & 19669\\
		\hline
		Glass recycling & T\textsubscript{2} & P\textsubscript{02}, P\textsubscript{E} & P\textsubscript{2} & 1.00, 175.50~\cite{Vossberg2014} & 1.00 & 39.06 & 3278\\
		\hline
		Metal recycling & T\textsubscript{3} & P\textsubscript{03}, P\textsubscript{E} & P\textsubscript{3} & 1.00, 2960~\cite{Duflou2015} & 1.00 & 1300.00 & 6629\\
		\hline
		Plastic recycling & T\textsubscript{4} & P\textsubscript{04}, P\textsubscript{E} & P\textsubscript{4}, P\textsubscript{06} & 1.00, 438~\cite{Hur2003} & 0.672, 0.328  & 364.50 & 9324\\
		\hline
		Composting~\cite{vanHaaren2010} & T\textsubscript{5} & P\textsubscript{05}, P\textsubscript{E} & P\textsubscript{5} & 1.00, 8.41 & 0.80  & 21.00 & 20470\\
		\hline
	\end{tabular}}
\end{table}

The drivers of the coordinated market model are the supply of waste streams and consumer demand for recycled products. For the purposes of this case study, our interest lies in municipal-scale transactions; a single supplier and consumer are modeled for each product, and the clearing of that transaction is synonymous with the clearing of the entire market for that product. A simple SC is defined consisting of a single node representing the city location ($N_{1}$), a node for the separation center ($N_{2}$), a node for the landfill ($N_{3}$), a node for the recycling service providers ($N_{4}$), and a node for the electrical power provider ($N_{5}$). The waste separation center splits MSW into its component streams which can be purchased by recycling service providers as process inputs. The landfill acts as a consumer that places a negative bid cost for municipal waste, indicating that it will take the city waste at a cost (tipping fee). Finally, recovered products from wast can be sold back to the city, which acts as a consumer. Note that the city acts as a supplier that offers MSW and as a consumer that requests electricity and recycled products (the city acts as two separate stakeholders). This notion is important, because a circular economy that fosters product recycling does not imply that there is a product cycle in the SC graph. The node-based and product-based representation of the MSW SC are presented in Figure~\ref{fig_network}. One of the advantages of the product-based representation  is that we are able to clearly represent material flows independently of the geographical arrangement of the stakeholders. Moreover, it is also clear that the SC does not contain technology cycles. The product-based representation provides a structure to manage SC complexity.
\begin{figure}[!htb]
	\center{\includegraphics[width=0.9\textwidth]{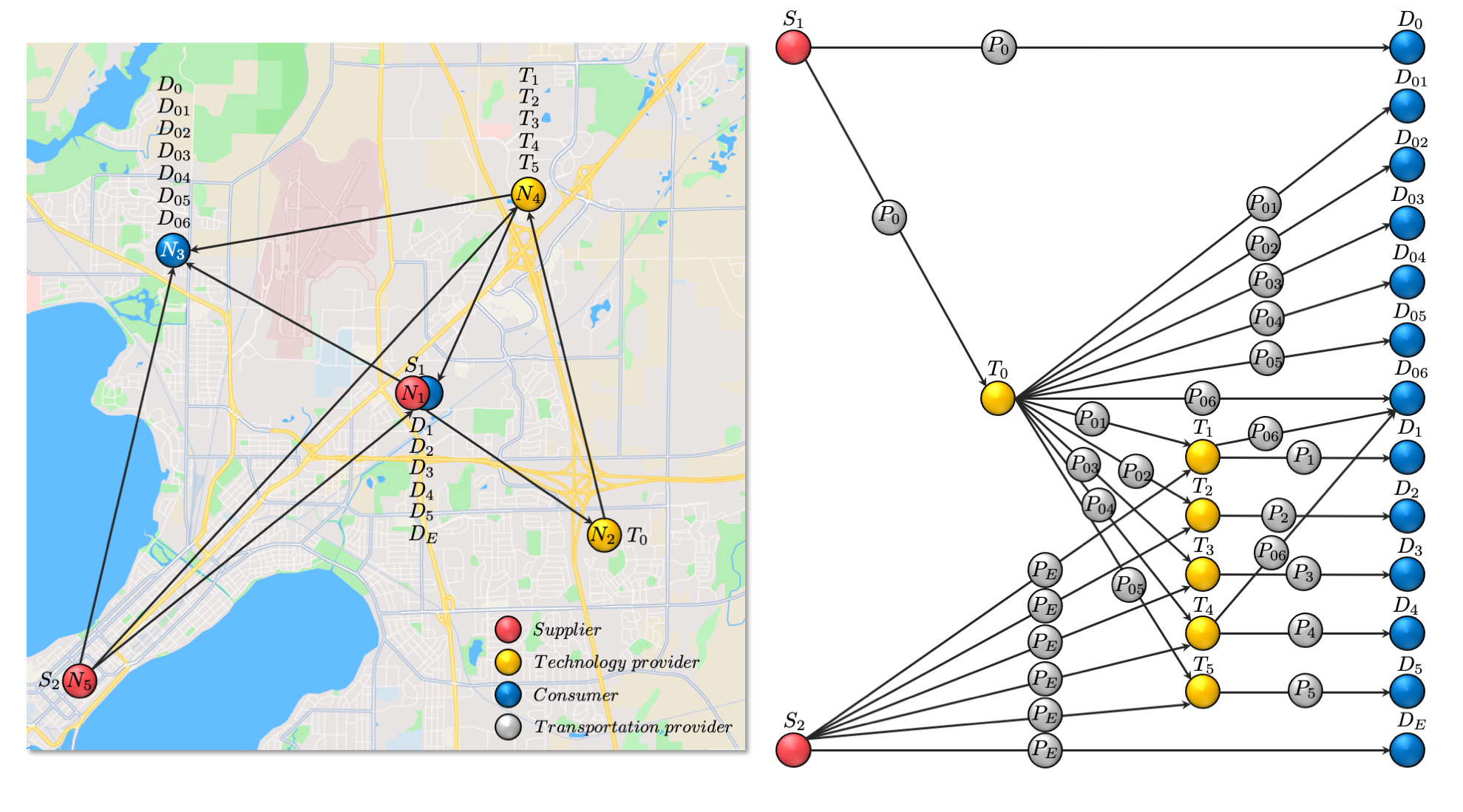}}
	\caption{Node-based SC (left) and product-based stakeholder graph (right) for illustration problem; stakeholders are abstracted into a DAG and are organized such that all transport flows are directed left-to-right following proper paths. This behavior is not clear in the SC map. Products carried by transporters are indicated on transportation nodes.}
	\label{fig_network}
\end{figure}

Stakeholder bid costs should provide incentives for the market to exist (for transactions to occur). Here, we assume that city residents offer a bid of -321 USD per tonne of MSW taken away (the negative value implying that they pay for this service)~\cite{HAweb}. Transportation costs for waste products, recycled products, and for compost are based on those in Garibay-Rodriguez \textit{et al.} and are applied on a per-tonne basis and according to product type~\cite{Garibay2018}. The separation service bid cost  is 235 USD per tonne~\cite{Santibanez2013}. The landfill service bid (tipping fee) is -57.12 USD per tonne~\cite{EREF2017} based on average US landfill fees (the negative bid value indicating the consumer requests payment to consume waste). The bids for all recycling service providers are taken to be their operating cost values listed in Table \ref{ProcessData}. Supply-side bids have been obtained or estimated using literature data; the problem remains to determine a set of demand bids which activate the MSW market and successfully divert waste from the landfill. The market-activating bids represent the lowest cost that the city residents need to pay for a given recycled product (recycled paper, glass, metal, plastic, compost, and electricity) to activate a market. This analysis seeks to reveal key recyclable products that drive the market.

The capacities associated with product supply and demand bids are presented in Table~\ref{CapacityData}; these values represent the maximum amount of a product that a consumer is willing to buy and a supplier is willing to provide. They are interpreted as upper limits on transaction volumes rather than provisions to be satisfied. The table is divided into demand, supply, and process capacity columns to indicate the products that are desired in the market and those that are available, as well as the transformation capacities. Landfill demand capacities are assumed to be sufficient to absorb all waste produced by the city. Additionally, the electrical grid is assumed to be able to provide sufficient power to satisfy both the city and recycling service provider requirements. Electricity consumption rates are estimated using EIA data~\cite{EIAFAQ} and supplier bid levels~\cite{EIA_Price}.

\begin{table}
	\centering
	\caption{Demand and supply capacity data; capacity values are in tonne per year except for electricity (see $S_{2}$ and $D_{E}$) which are indicated in GWh per year}
	\label{CapacityData}
	\begin{tabular}{cccrrr} 
		\hline
		\textbf{Stakeholder}&﻿\textbf{Node} & \textbf{Product} & \textbf{\makecell{Supply\\Capacity}} &\textbf{\makecell{Demand\\Capacity}} & \textbf{\makecell{Process\\Capacity}}\\\hline
		City ($S_{1}$)&$N_{1}$& P\textsubscript{0} &72,847& - &-\\\hline
		Electrical grid ($S_{2}$)&$N_{5}$& P\textsubscript{E} &1,070& - &-\\\hline
		City ($D_{1}$)&$N_{1}$& P\textsubscript{1} & - &16,522&-\\\hline
		City ($D_{2}$)&$N_{1}$& P\textsubscript{2} & - &3,279&-\\\hline
		City ($D_{3}$)&$N_{1}$& P\textsubscript{3} & - &6,630&-\\\hline
		City ($D_{4}$)&$N_{1}$& P\textsubscript{4} & - &6,267&-\\\hline
		City ($D_{5}$)&$N_{1}$& P\textsubscript{5} & - &16,376&-\\\hline
		City ($D_{E}$)&$N_{1}$& P\textsubscript{E} & - &1,040&-\\\hline
		Landfill ($D_{0}$)&$N_{3}$& P\textsubscript{0} & - &72,847&-\\\hline
		Landfill ($D_{01}$)&$N_{3}$& P\textsubscript{01} & - &19,669&-\\\hline
		Landfill ($D_{02}$)&$N_{3}$& P\textsubscript{02} & - &3,278&-\\\hline
		Landfill ($D_{03}$)&$N_{3}$& P\textsubscript{03} & - &6,629&-\\\hline
		Landfill ($D_{04}$)&$N_{3}$& P\textsubscript{04} & - &9,324&-\\\hline
		Landfill ($D_{05}$)&$N_{3}$& P\textsubscript{05} & - &20,470&-\\\hline
		Landfill ($D_{06}$)&$N_{3}$& P\textsubscript{06} & - &20,630&-\\\hline
		Separator ($T_{0}$)&$N_{2}$& P\textsubscript{0} & - & - &72,847\\\hline
		Paper recycling ($T_{1}$)&$N_{4}$& P\textsubscript{01} & - & - &19,700\\\hline
		Glass recycling ($T_{2}$)&$N_{4}$& P\textsubscript{02} & - & - &3,278\\\hline
		Metal recycling ($T_{3}$)&$N_{4}$& P\textsubscript{03} & - & - &6,629\\\hline
		Plastic recycling ($T_{4}$)&$N_{4}$& P\textsubscript{04} & - & - &9,324\\\hline
		Composting ($T_{5}$)&$N_{4}$& P\textsubscript{05} & - & - &2,0470\\\hline
	\end{tabular}
\end{table}

\subsection{Market Activating Bids}

Our results consist of nine case studies in which we vary the city demand bids for the five recycled products (paper, glass, metal, plastic, and compost, corresponding to the demands $D_{1}$ to $D_{5}$ in Table~\ref{CapacityData}). The demand bid for electricity is generally maintained at the Energy Information Administration (EIA) value including transmission costs. The responses of interest in these cases are stakeholder profits (presented in aggregate for transport services) product volumes exchanged, and product clearing prices. This information is collected in Table~\ref{CaseStudyResults}. Each column in the table corresponds to one scenario and includes the objective function value (representing the scenario social welfare) the bid values corresponding to the recycled product demands $D_{1}$ to $D_{5}$ and the city electricity demand $D_{E}$. We report the profits associated with demand, supply, transport, and technology stakeholders, demand allocations, and the clearing prices at the city node $N_{1}$. We also calculate the percentage of waste recycled in each scenario, which is interpreted as the annual waste diversion from landfill disposal. The nine presented cases are selected to demonstrate properties of the coordinated market framework, or to provide contrast with scenarios in which stakeholders are forced to participate in a market. The cases are numbered for reference. We discuss the motivation for and observations resulting from each case in turn.

Case 1 demonstrates the use of the market clearing framework using a selection of bids sufficiently high to activate all the recycling pathways. The bid values are set sufficiently high (either 1,500 USD per tonne, or 2,000 USD per tonne) that each recycling pathway is activated, but because we do not apply a systematic approach to bid selection in this case, market profits (and social welfare) are large. In this case, the electricity supplier, separation technology, and compost consumer all have large associated profit values. Note that in the case of the consumer, its profit is counted as money it did not need to spend (meaning that the market price for $P_{5}$ (compost) is lower than the $D_{5}$ bid. Checking the table confirms this observation; the bid for compost is 1,500 USD per tonne, while the market price for compost is only 43.70 USD per tonne. Clearly, the bid of 1,500 USD per tonne is inefficient in the face of this market price. This case illustrates an important property of the market clearing framework: profit allocation is degenerate (has multiple solutions), and the ISO does not always produce a ``fair" allocation of profits to stakeholders \cite{Sampat2018A}. In this case, profits are allocated to four stakeholders, and all the others get nothing. Issues associated with fairness will be explored in future work. 

Moving on to case 2, we decide to lower stakeholder bids in an attempt to obtain a more efficient set of market prices. Here, we set (again, with no real methodology to our choice) recycled product bids to 500 USD per tonne, and we also decrease the city's bid for electrical power to 0.10 USD per kWh. The outcome of this bidding setting is that the {\em recycling market is completely dry}; all waste is shipped to the landfill, and moreover, city residents are without power as well. The clearing prices for recycled products satisfy the bounds (the bids) we set, but these bids are too low to activate the markets.

In case 3, we will attempt to activate markets using a {\em constraint that forces all the recycling technologies to be active}. To be specific, we constrain the demands $D_{1}$ to $D_{5}$ and $D_{E}$ to be fulfilled to their capacity values. In more realistic terms, the city municipal officials might mandate that waste is to be recycled (to achieve landfill diversion). Further, since policy has been used to obtain the solution we wanted, we will not increase our bids from the 500 USD per tonne level, nor from the 0.10 USD per kWh level for power. From the allocation data in the table, we see that {\em this approach achieves the same results as case 1}. Specifically, we have achieved 71\% annual landfill diversion (a fair degree of circularity, based on diversion alone) and at a lower bid cost, too. However, examination of stakeholder profits reveals that three of the recycling technologies, as well as the city power supplier, are operating at a loss (with large negative profits). By constraining the market model to force the participation of certain stakeholders, we allow the ISO to allocate negative profits, and lose a useful model property, as we determined will occur in Section \ref{sec: forced markets}. In other words, our policy solution has resulted in stakeholders operating at net annual loss. This is not a financially sustainable strategy for long-term waste management; presumably the recycling services will soon go out of business because they are not profitable (or would need to be subsidized externally). This result confirms our observations that forcing market participation through constraints destroys the economic properties of the market clearing model. 

Cases 1 to 3 motivate the need for an approach to determine an economically sound mechanism to activate markets. We thus apply our market-activating bid approach to generate further results. In case 4, we use our approach to determine the market activating bids $\underline{\alpha}^{d}_{p}$ according to the procedure outlined in section \ref{MinClearBids} for consumers $D_{1}$ to $D_{5}$ such that each consumer bids enough to cover the processing and transportation costs associated with the product; e.g., the market activating bid associated with $D_{1}$ is calculated based on technology $T_{1}$, the disposal cost for non-recyclable waste resulting from $T_{1}$, and the transportation costs for product $P_{1}$. We repeat this calculation for each of $D_{1}$ to $D_{5}$, and present the results in column 4 as partial market activating bids. Solving the clearing model, we observe that these bids result in a dry market once again, because in calculating them we did not account for the cost of separation, and the city waste disposal bid is not sufficient to completely cover the separation cost, either. Since there is not sufficient revenue available to the ISO to activate the market, the outcome is dry. Nevertheless, this result provides an opportunity: we use Theorem \ref{ThmSharing} to distribute the cost of waste separation between the portion covered by the city's waste disposal bid and any of the recycled product consumers that we choose. To illustrate how this selection will influence the market we choose to calculate the additional bidding value required from each stakeholder $D_{1}$ to $D_{5}$ in turn, and we present these as cases 5 to 9 in which $D_{1}$ to $D_{5}$ cover the remaining cost of separation. These columns are labeled as "covering cases" in the table for shorthand. We elect to put the burden of the separation cost on the recycled products rather than increasing the city bid for waste disposal because it allows us to demonstrate distribution of remuneration to multiple stakeholders. To illustrate the market activating bids, we demonstrate the calculations required to determine the bid for the demand $D1$ in case 5 in Figure~\ref{fig_DemoCalcs}. This figure illustrates the calculations leading to values of 66.30 USD per tonne MSW for city, and for the bid of 1165.62 USD per tonne paper in case 5.
\begin{figure}[!htb]
	\center{\includegraphics[width=0.9\textwidth]{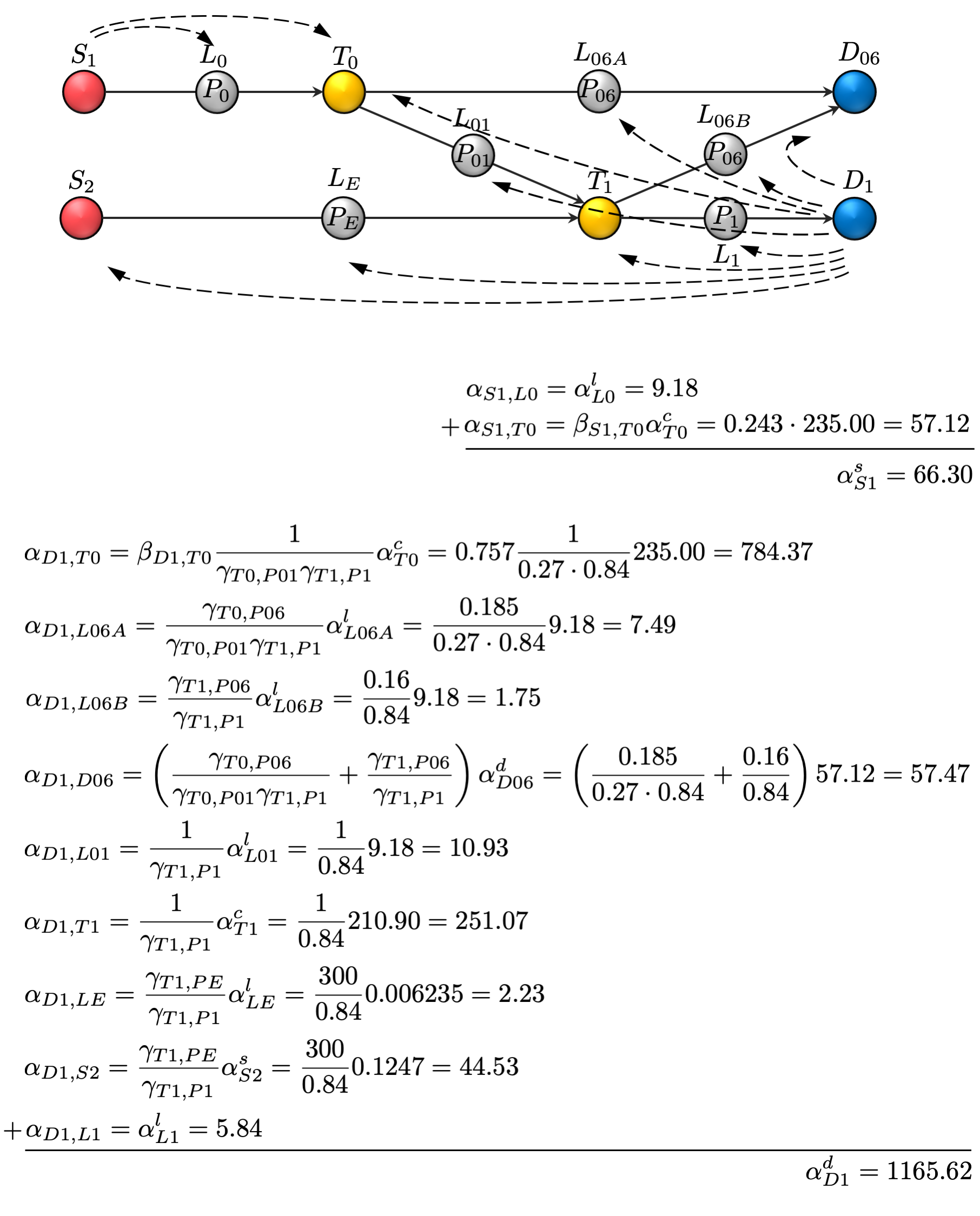}}
	\caption{Calculation procedure for the $D1$ market-activating bid. The relevant SC substructure from Figure~\ref{fig_network} is included for illustration with arrows indicating the relevant remuneration from the supplier $S_1$ and the consumer $D_1$. Note that both stakeholders share remuneration of technology $T_0$}
	\label{fig_DemoCalcs}
\end{figure}

We present the bidding and solution data in cases 5 to 9 in table~\ref{CaseStudyResults}. We emphasize the covering stakeholder bid in bold, while the other bids remain the same. We note that all five of these cases result in cleared markets, and in each case all recycling markets are active, so the annual diversion rate is 71\% across the board. Importantly, any profits generated in these cases are small. Consider, for example, the profit of 69.87 USD allocated to the separation service in case 5; this value is small relative to the 78,247 tonnes of waste separated. Recalling our definition of profit from \eqref{Profits}, the price bounding properties in Section \ref{sec: properties}, and the revenue adequacy interpretation of market-activating bids in Theorem~\ref{BERevAqc}, we conclude that achieving small profits is synonymous with bounding market prices near stakeholder bids. Examination of the market prices in the results demonstrates that we have achieved our goal; our market-activating bid methodology results in bids that bound market prices closely; in our results the bid values we present are equal to the prices observed at least to the second decimal place (we need to allow a small profit to  exist to break social welfare degeneracy with the dry case). An important interpretation of these results is that when we say that profits are small, we mean profits over and above the remuneration rate defined by stakeholder bids. Cases 5 to 9 are characterized by small profits. This does not mean that stakeholders are not making money. This means that stakeholders are paying (or being paid) the amount they bid. These results demonstrate the revenue adequacy property. Our methodology results in an efficient outcome for all stakeholders (none are paying more than is required to keep the market active, nor receiving less than they need to participate in the market) and our methodology mitigates pricing uncertainty for stakeholders.

Contrasting cases 2 and 5 elucidates useful properties of the market-activating bid approach. In case 2 we set bids of 500.00 USD per tonne for each recycled product, and the resulting market was dry. In case 5, two of the bids we calculated are lower than the 500.00 USD per tonne equivalent in case 2. This result illustrates the SC dependence of the calculations. The market-activating bids required from each stakeholder are a function of the network structure of the supply chain. This also means that the total combined bid cost required to activate markets is not a single value but is also dependent on SC structure. For instance, in case 5 the bid $\alpha^{d}_{P1}$ is increased from 327.24 USD per tonne up to 1,165.62 USD per tonne, a difference of 838.38 USD per tonne. In contrast, in case 7 the bid $\alpha^{d}_{P3}$ increases from 1,702.59 USD per tonne to 3,792.10 USD per tonne, an increase of 2,089.51 USD per tonne. How we choose to activate markets (i.e., which stakeholder we assign the burden of the separation cost in our example) results in a different bid cost. These differences are a function of SC structure and technology yields. Moreover, these results demonstrate that upstream processes have a different value to different stakeholders, which we are able to capture through market prices and the bidding procedure. Consequently, we observe that we can activate markets using any of the available products, but market prices reflect that it is more expensive to activate the market with certain products.

The reader may observe a consistent trend in the results: the city, as a waste supplier, always has a profit of at least 1.86\e{7} USD. The reader might also observe that the social welfare value is also 1.86\e{7} USD. Indeed, under the interpretation of the market clearing objective function as the total of stakeholder profits, we arrive at an interesting result. Case 2, the scenario in which recycling markets are dry, has all waste sent to the landfill. Under the selected waste supplier bid of -321 USD per tonne, and the landfill tipping fee of -57.12 USD per tonne, and waste transportation cost of 9.18 USD per tonne, the resulting profit available is (321.00 - 57.12 - 9.18) = 254.70 USD per tonne waste. Multiplying this value by 72,847 tonnes we obtain 1.86\e{7} USD, as observed as the social welfare value. This result is interesting because it sets a baseline that the recycling market must be able to exceed. If the social welfare value of the market with active recycling is less than that where all waste is landfilled, then the landfill solution will be preferred according to the social welfare metric. Moreover, our market activating bid methodology allows us to determine bid levels that result in zero profit to stakeholders. This is a problem, since the social welfare of the recycling market must be greater than the social welfare of landfilling if we want to incentivize recycling. To ensure that the recycling market provides at least as much social welfare as the landfill solution, when calculating market-activating bids we assume that the waste supplier $S_{1}$ will not need to pay more than it would in the landfill solution; i.e., the city's contribution to paying for waste transport and separation will not exceed 57.12 + 9.18 = 66.30 USD per tonne. This requirement is reflected in the calculation illustrated in Figure~\ref{fig_DemoCalcs}. The profits made available this way are required in order to ensure that the market attains the necessary social welfare value. If we did not make the profit available to be allocated to the supplier (or another stakeholder) we would instead need to calculate market activating bids with an extra incentive value that makes the necessary profits available to be allocated by the ISO. Thus in every case the social welfare is at least 1.86\e{7} USD; because this value is always available from the landfill solution, the recycling market must be able to exceed it.

\begin{table}.
	\centering
	\caption{Results for coordinated market case studies.}
	\label{CaseStudyResults}
	\resizebox{\textwidth}{!}{%
		\begin{tabular}{|c|ccc|cccccc|} 
			\hline
			&\textbf{\makecell{1. Clear\\Coordinated}}&\textbf{\makecell{2. Dry\\Coordinated}}&\textbf{\makecell{3. Forced}}&\textbf{\makecell{4. Partial Bids\\(Dry)}}&\textbf{5. P1 Covering}&\textbf{6. P2 Covering}&\textbf{7. P3 Covering}&\textbf{8. P4 Covering}&\textbf{9. P5 Covering}\\\hline\hline
			Objective (USD)&&&&&&&&&
			\\\hline
			&7.95\e{7}&1.86E+07&-2.48\e{7}&1.86\e{7}&1.86\e{7}&1.86\e{7}&1.86\e{7}&1.86\e{7}&1.86\e{7}
			\\\hline\hline
			Bids (USD/tonne)&&&&&&&&&
			\\\hline\hline
			$\alpha^{d}_{D1}$&1,500.00&500.00&500.00&327.24&\textbf{1165.62}&327.24&327.24&327.24&327.24
			\\\hline
			$\alpha^{d}_{D2}$&1,500.00&500.00&500.00&77.06&77.06&\textbf{4302.52}&77.06&77.06&77.06
			\\\hline
			$\alpha^{d}_{D3}$&2,000.00&500.00&500.00&1,702.59&1,702.59&1,702.59&\textbf{3,792.10}&1,702.59&1,702.59
			\\\hline
			$\alpha^{d}_{D4}$&1,500.00&500.00&500.00&679.62&679.62&679.62&679.62&\textbf{2,890.20}&679.62
			\\\hline
			$\alpha^{d}_{D5}$&1,500.00&500.00&500.00&43.51&43.51&43.51&43.51&43.51&\textbf{889.35}
			\\\hline
			$\alpha^{d}_{DE}$&0.15&0.10&0.10&0.13&0.13&0.13&0.13&0.13&0.13
			\\\hline\hline
			Profits (USD)&&&&&&&&&
			\\\hline\hline
			City (P0 supply)&1.86\e{7}&1.86\e{7}&2.34\e{7}&1.86\e{7}&1.86\e{7}&1.86\e{7}&1.86\e{7}&1.86\e{7}&1.86\e{7}
			\\\hline
			Grid (PE supply)&2.04\e{7}&&&&0.25&0.25&&&
			\\\hline
			Landfill (P0 demand)&&&&&&&&&
			\\\hline
			Transport&&&&&&&&&
			\\\hline
			Separation (T0)&1.67\e{7}&&&&69.87&55.34&39.44&15.91&
			\\\hline
			Paper recycling (T1)&&&&&&&&&
			\\\hline
			Glass recycling (T2)&&&&&&&&&
			\\\hline
			Metal recycling (T3)&&&&&&&&&
			\\\hline
			Plastic recycling (T4)&&&&&&&&&
			\\\hline
			Composting (T5)&&&&&&&&&
			\\\hline
			City (P1 demand)&&&2.85\e{6}&&&&&&
			\\\hline
			City (P2 demand)&&&-1.73\e{7}&&&&&&
			\\\hline
			City (P3 demand)&&&-7.97\e{6}&&&&&&
			\\\hline
			City (P4 demand)&&&-1.13\e{6}&&&&&&
			\\\hline
			City (P5 demand)&2.38\e{7}&&7.48\e{6}&&&&5.47\e{-5}&5.47\e{-5}&15.91
			\\\hline
			City (PE demand)&&&-3.22\e{7}&1,039.90&1,039.65&1,039.65&1,039.98&1,039.98&1,039.98
			\\\hline\hline
			Demand Allocations&&&&&&&&&
			\\\hline\hline
			Landfill, P0 (tonne)&&72,847&&72,847&&&&&
			\\\hline
			Landfill, P06 (tonne)&19,682&&19,682&&19,682&19,682&19,682&19,682&19,682
			\\\hline
			City, P1 (tonne)&16,521&&16,521&&16,521&16,521&16,521&16,521&16,521
			\\\hline
			City, P2 (tonne)&3,278&&3,278&&3,278&3,278&3,278&3,278&3,278
			\\\hline
			City, P3 (tonne)&6,629&&6,629&&6,629&6,629&6,629&6,629&6,629
			\\\hline
			City, P4 (tonne)&6,266&&6,266&&6,266&6,266&6,266&6,266&6,266
			\\\hline
			City, P5 (tonne)&16,376&&16,376&&16,376&16,376&16,376&16,376&16,376
			\\\hline
			City, PE (GWh)&1,040&&1,040&1,040&1,040&1,040&1,040&1,040&1,040
			\\\hline\hline
			Recycled waste (\%)&&&&&&&&&
			\\\hline
			&71&&71&&71&71&71&71&71
			\\\hline\hline
			City ($N_{1}$) clearing prices&&&&&&&&&
			\\\hline\hline
			P0 (USD/tonne)&-66.30&-66.30&&-66.30&-66.30&-66.30&-66.30&-66.30&-66.30
			\\\hline
			P1 (USD/tonne)&1,500.00&500.00&330.47&327.24&1165.62&327.24&327.24&327.24&327.24
			\\\hline
			P2 (USD/tonne)&1,500.00&500.00&5,777.44&77.06&77.06&4,302.52&77.06&77.06&77.06
			\\\hline
			P3 (USD/tonne)&2,000.00&500.00&1,729.42&1,702.59&1,702.59&1,702.59&3,792.10&1,702.59&1,702.59
			\\\hline
			P4 (USD/tonne)&1,500.00&500.00&685.52&679.62&679.62&679.62&679.62&2,890.20&679.62
			\\\hline
			P5 (USD/tonne)&43.70&500.00&43.60&43.51&43.51&43.51&43.51&43.51&889.34
			\\\hline
			PE (USD/kWh)&0.15&0.13&0.13&0.13&0.13&0.13&0.13&0.13&0.13\\\hline
	\end{tabular}}
\end{table}

\subsection{Drawing Comparisons}

As a final evaluation of our results, we examine the City of Madison, which has similar-sized population to our example problem, and two active landfills \cite{MadisonLandfills}. The city operating budget for 2018 reports about 9.2 million USD in expected solid waste management expenses. We note that our 1,000,000 person model predicts a landfill revenue stream of 4.83 million. Accounting for the population difference suggests a landfill cost of about 12.6 million USD. Our predictions for landfill cost are on a similar order of magnitude.

Comparison brings to mind an important difference between the city's budget and our results: our numbers represent cases in which either all waste is disposed of via landfilling, or where we incentivize the maximum use of recycling to offset landfill use. In our case studies, 71\% annual landfill diversion is achievable (indicating that a large fraction of MSW can be circularized) and the remaining fraction of waste is reckoned to consist of either non-recyclable materials or materials for which a recycling program is not available. Importantly, this means our (scale adjusted) estimate of landfill use cost represents an overestimate of the City of Madison cost. Scaling  by the population, our model suggests a cost 12.6 million USD cost for landfill use. Our overestimate may reflect the fact that the City of Madison maintains an active recycling program of its own and is not using the landfill to dispose of the entirety of the waste produced in the city.

Our case studies suggest that the total value of recycling (as measured in consumer revenue streams) is about 36 million USD over all five recycling services. Adjusting for population, this would suggest the social value of recycling all possible waste produced by a city the size of Madison, WI would be something on the order of 93 million USD. This estimate provides an indicator of the value of the materials that can be recovered from waste, but also suggest the value that these products must have in order to incentivize and maintain the activity of recycling services.

Commodity prices provide a means of evaluating the market-activating bids we determined. The June 1, 2020 spot price for aluminum is 1537.00 USD per tonne \cite{BloombergAluminum}; this price is slightly higher than the consumer bid required in our case studies, but it is not sufficiently high if metals are to be penalized to cover SC MSW separation costs. An estimate of the current HDPE price (as a comparator for plastic) is obtained at 1141 USD per tonne \cite{PlasticPrice}. This estimate is well in excess of the 679.62 USD per tonne market price that we determine, but once again, falls short of the 2,890.20 USD per tonne that we determine with the separation cost burden. These examples suggest that our methodology and 100,000 person case study produced reasonable values. The market data suggest that SC costs (like upstream separation) may need to be distributed over multiple products in order to keep individual product prices near the prices of similar commodities.

\section{Conclusions}

We have derived economic properties for multi-product supply chains by using a coordinated market interpretation. The Lagrangian and dual representations of the market clearing problem unravel fundamental economic properties (competitiveness and revenue adequacy) and reveal proper remuneration mechanisms for stakeholders and pricing behavior. This analysis shows that in a coordinated market, no stakeholder can incur a financial loss. The dual representation also reveals that activating markets by forcing stakeholder participation destroys such a guarantee and introduces arbitrary price behavior. We presented a new product-based representation of a supply chain (we call this a stakeholder graph) that captures interconnectivity of stakeholders via products. This representation reveals market topology information that is not obvious from the traditional node-based (geographical) representation.  The stakeholder graph reveals that it is  unlikely for technology cycles to occur and thus its topology is expected to have a directed acyclic graph structure. This structure allows us to derive a computational procedure to calculate minimum bid costs that activate the market prior to solving the market clearing problem. This market-activating analysis can be used to determine market existence criteria (in the form of lower bounds on prices and bids) and can be used to guide policy that seeks to incentivize markets. Notably, this analysis only relies on topological information of the supply chain. We demonstrate the derived concepts using a case study for municipal waste management. Here, we use our results to determine consumer bids that will activate recycling and that will divert waste from the landfill.  As part of future work, we would like to perform deeper analysis on existence of technology cycles in stakeholder graphs. We are also interested in extending out framework to time-dependent supply chains with inventories and to account for fairness issues. 

\bibliography{EconomicPropertiesBib.bib}


\section*{Acknowledgments}

We acknowledge support from the U.S. Department of Agriculture (grant 2017-67003-26055) and partial funding from the National Science Foundation (under grant CBET-1604374).


\end{document}